\documentclass[11pt,a4paper]{article}
\usepackage{amssymb,amsmath,amsfonts,mathrsfs,bm,enumerate,comment}
\allowdisplaybreaks[1]
\numberwithin{equation}{section}

\usepackage[colorlinks=true, pdfstartview=FitV, linkcolor=blue, citecolor=blue, urlcolor=blue,pagebackref=false]{hyperref}

\usepackage{tikz}
\usepackage{float}
\usepackage[font=footnotesize]{caption}

\usepackage{times,theorem,latexsym,color,comment}

\newcommand{\BOX}{\ensuremath\Box}

\newtheorem{theorem}{Theorem }[section]

{\theorembodyfont{\rmfamily}}
{\theorembodyfont{\rmfamily}}
{\theorembodyfont{\rmfamily}}
\newtheorem{lemma}[theorem]{Lemma}
\newtheorem{proposition}[theorem]{Proposition}
{\theorembodyfont{\rmfamily}\newtheorem{remark}[theorem]{Remark}}
{\theorembodyfont{\rmfamily}}

\newcommand{\N}{\mathbb{N}}

\newcommand{\R}{\mathbb{R}}
\newcommand{\C}{\mathbb{C}}

\newcommand{\dd}{\,d}
\newcommand{\opspt}{\operatorname{spt}}

\newcommand{\opRan}{\operatorname{Ran}}
\newcommand{\opdiv}{\operatorname{div}}
\newcommand{\opDiv}{\operatorname{Div}}
\newcommand{\oprot}{\operatorname{rot}}

\newcommand{\eps}{\varepsilon}
\newcommand{\del}{\delta}

\newcommand{\oparg}{\operatorname{arg}}

\newcommand{\ii}{\mathrm i}

\newcommand{\overbar}[1]{\mkern 1.5mu\overline{\mkern-1.5mu#1\mkern-1.5mu}\mkern 1.5mu}

\def\XXint#1#2#3{{\setbox0=\hbox{$#1{#2#3}{\int}$}
		\vcenter{\hbox{$#2#3$}}\kern-.5\wd0}}

\newenvironment{proof}{{\vskip\baselineskip\noindent\textbf{Proof:}}}
{\hspace*{.1pt}\hspace*{\fill}\BOX\vskip\baselineskip}

\newenvironment{proofx}[1]
{\vskip\baselineskip\noindent\textbf{Proof of {#1}:}}
{\hspace*{.1pt}\hspace*{\fill}\BOX\vskip\baselineskip}

\begin{document}

\title{A Runge-type theorem by remote forcing \\
for the linearized resistive MHD system}

\author{
Mitsuo Higaki
\thanks{
Department of Mathematics, 
Graduate School of Science, 
Kobe University, 
1-1 Rokkodai, Nada-ku, Kobe 657-8501, Japan.
\textit{E-mail address:}\texttt{higaki@math.kobe-u.ac.jp}
}
\and
Franck Sueur
\thanks{
Department of Mathematics, 
Maison du nombre, 6 avenue de la Fonte, 
University of Luxembourg, 
L-4364 Esch-sur-Alzette, Luxembourg.
\textit{E-mail address:}\texttt{Franck.Sueur@uni.lu}
}
}
\date{}

\maketitle

\begin{abstract}
In this paper, we study a quantitative Runge-type global approximation theorem for the linearized magnetohydrodynamic (MHD) system in bounded domains with arbitrary topology. In the context of magnetic relaxation, the interplay between the domain topology and magnetic field structure plays a crucial role. Recent studies illustrate a sharp contrast in the dynamics: while Enciso--Peralta-Salas \cite{EncisoPeralta-Salas2025} highlights that the geometric complexity of magnetic fields acts as an obstruction to relaxation in non-resistive regimes, Kozono-Shimizu-Yanagisawa \cite{KozonoShimizuYanagisawa2025} proves that in resistive regimes, the flow stably relaxes towards a harmonic equilibrium. Focusing on this resistive scenario, we adopt a control-theoretic viewpoint to quantitatively approximate the relaxation trajectory generated by the linearized initial-boundary value problem. Specifically, after decomposing the bounded-domain solution into the time-evolving part and the stationary part, we approximate it by a global solution on $\R^3$ under a remote forcing. An explicit dependence of the forcing cost on the approximation error is provided. 
\end{abstract}

\tableofcontents

    \section{Introduction and main results}

The interplay between the topology of the underlying domain and the geometric structure of magnetic fields is a central theme in modern magnetohydrodynamics (MHD). A fundamental question in this field, dating back to the pioneering works of Arnold \cite{Arnold1966}, Parker \cite{Parker1972,Parker1994_book}, and Moffatt \cite{Moffatt1985,Moffatt1986}, is the problem of magnetic relaxation: determining whether and how a given magnetic field evolves towards an MHD equilibrium.

Recent mathematical results have revealed that the answer depends critically on the presence of resistivity. In the context of the non-resistive MHD system, Enciso--Peralta-Salas \cite{EncisoPeralta-Salas2025} identified obstructions to magnetic relaxation. They proved that generic magnetic fields, due to the conservation of complex geometric structures, cannot be transformed into equilibria via volume-preserving diffeomorphisms. This highlights the rigidity of magnetic field lines in the ideal MHD case, where they are transported by the fluid flow.

In sharp contrast, for the resistive MHD system, the presence of magnetic diffusivity acts as a mechanism for energy dissipation, breaking the geometric rigidity in the non-resistive case. Focusing on this dissipative regime, Kozono-Shimizu-Yanagisawa \cite{KozonoShimizuYanagisawa2025} established the exponential stability of harmonic magnetic fields. Although their result relies on a smallness assumption to utilize perturbation techniques, it reveals a fundamental topological selection principle: the flow does not simply decay to zero, but stably relaxes towards a non-trivial harmonic equilibrium dictated by the domain's cohomology (see \cite{Schwarz1995_book}).

From a control-theoretic viewpoint, it is a highly non-trivial question whether such a physically natural, dissipative relaxation process can be artificially replicated by remote forcing. In this paper, we focus on the linearized resistive MHD system and investigate the global approximation of such a relaxation process. The global approximation theorem, or the Runge property, is a cornerstone of the theory of partial differential equations, asserting that a solution defined on a subdomain can be approximated by one defined on a larger domain. While the qualitative aspects of this property have been well established through the classic works for elliptic equations \cite{Lax1956,Malgrange1955,Browder1962} and the developments in parabolic equations \cite{Jones1975,EncisoGarcía-FerreroPeralta-Salas2019}, recent research has shifted towards its quantitative aspect to determine the ``cost of approximation'' \cite{RülandSalo2019,EncisoPeralta-Salas2021,DebrouwereKalmes2025}. A compelling example of this shift is found in the study of topological transitions: Enciso--Peralta-Salas \cite{EncisoPeralta-Salas2021} leveraged a quantitative Runge-type theorem for the nonlinear Schr\"{o}dinger equation to rigorously demonstrate the existence of quantum vortex reconnection. Their work highlights that the ability to precisely control the approximation error is not merely technical, but essential for capturing such structural evolutions. Inspired by this perspective, specifically the quantitative control of topological dynamics, we aim to evaluate the cost of replicating the aforementioned magnetic relaxation. Our main result, Theorem \ref{thm.main} below, provides explicit estimates that account for the hybrid nature of the dynamics, decomposed into a dissipative time-evolving part and a stationary topological part, thereby quantifying the control cost for replicating the magnetic relaxation towards an equilibrium. Although our analysis is restricted to the linearized system, this quantitative framework serves as a crucial first step towards controlling the full nonlinear MHD system.

Regarding the time-evolving part above, our analysis builds upon the $L^2$-quantitative Runge approximation for the non-stationary Stokes system obtained in \cite{Higaki2025_Runge}. However, that result was restricted to a bounded domain $D$ whose complement in $\R^3$ is connected; that is, whose second Betti number $\beta_2(D)$ is zero by Alexander duality \cite[Section 3.3]{Hatcher2002_book}. In this paper, we extend this quantitative theory to the linearized resistive MHD system in a domain $D$ with arbitrary topology, allowing for a non-zero second Betti number.

The mathematical foundation of the MHD system was laid by works such as Duvaut-Lions \cite{DuvautLions1972} and Sermange-Temam \cite{SermangeTemam1983}. When the first Betti number $\beta_1(D)$ is positive, the space $X_{{\rm har}}(D)$ of Neumann harmonic fields is non-trivial and represents the first cohomology of $D$. These fields encode the circulation around the handles of the domain. Consequently, the operator associated with the following linearized MHD system around $(0,H_*)$ possesses a non-trivial kernel, where $H_*$ is a background harmonic magnetic field:
\begin{equation*}\tag{MHD}\label{eq.MHD}
    \left\{
    \begin{array}{ll}
    \partial_{t} v - \Delta v 
    - H_*\cdot\nabla H - H\cdot\nabla H_* 
    + \nabla q
    = 0&\mbox{in}\ D\times (0,\infty), \\
    \partial_{t} H + \oprot (\oprot H) + \oprot (H_* \times v)
    = 0&\mbox{in}\ D\times (0,\infty), \\
    \opdiv v
    = \opdiv H 
    = 0&\mbox{in}\ D\times [0,\infty), \\
    (v,\, H\cdot n,\, \oprot H \times n) 
    = (0,0,0)&\mbox{on}\ \partial D\times (0,\infty), \\
    (v, H)
    = (v_0, H_0)\in L^2_\sigma(D)\times L^2_\sigma(D) &\mbox{on}\ D\times \{0\}. 
    \end{array}\right. 
\end{equation*}

    \paragraph{Notation and functional setting.}

From this point on, $D\subset \R^3$ is a bounded (connected) domain with smooth boundary, and $n$ denotes the unit outward normal on $\partial D$. We write
\[
    \partial D = \Gamma_0 \cup \cdots \cup \Gamma_M
\]
for the decomposition of $\partial D$ into its connected components. For an open set $G\subset \R^3$ and complex-valued mappings $f,g\in L^2(G)^m$ with $m\ge1$, we set 
\begin{equation}
    \langle f,g\rangle_G
    =
    \int_G f\cdot \overbar g \dd x.
\end{equation}
For every bounded Lipschitz domain $G\subset \R^3$, we write
\begin{equation}
    L^2_\sigma(G)
    :=
    \{
    u\in L^2(G)^3 
    \mid 
    \opdiv u = 0 \text{ in } \mathcal D'(G),\ 
    \gamma_n u = 0 \text{ in } H^{-1/2}(\partial G)
    \},
\end{equation}
where $\gamma_n$ denotes the normal trace on $H(\opdiv;G)=\{u\in L^2(G)^3 \mid \opdiv u\in L^2(G)\}$; see \cite[Chapter I, Section 2]{GiraultRaviart1986} for details. For our fixed domain $D$, we further set
\begin{equation}
\begin{aligned}
    X_{{\rm har}}(D)
    &=
    \{h\in C^\infty(\overbar{D})^3 \mid
    \oprot h = 0,\ \opdiv h = 0,\ h\cdot n|_{\partial D}=0\},\\
    \mathscr X^2_\sigma(D)
    &=
    X_{{\rm har}}(D)^\perp \cap L^2_\sigma(D).
\end{aligned}
\end{equation}
Then the Helmholtz-Weyl decomposition gives
\begin{equation}
    L^2_\sigma(D)
    =
    X_{{\rm har}}(D)\oplus \mathscr X^2_\sigma(D).
\end{equation}
In particular, $\dim X_{{\rm har}}(D)=\beta_1(D)$, whereas $\beta_2(D)$ represents the number of the bounded connected components of $\R^3\setminus \overbar{D}$ through Alexander duality; see \cite[Section 3.3]{Hatcher2002_book}. We denote by $\mathbb{P}$ the Helmholtz projection from $L^2(D)^3$ onto $L^2_\sigma(D)$, and by $\mathbb{P}_G$ the corresponding projection defined on a general domain $G$. Moreover, for an open set $G$ with Lipschitz boundary decomposed into connected components (hence domains) $G_1,\ldots,G_m$, we define
\[
    \mathbb{P}_G
    =
    \mathbb{P}_{G_1}+\cdots+\mathbb{P}_{G_m},
    \qquad
    L^2_\sigma(G)
    =
    L^2_\sigma(G_1)\oplus\cdots\oplus L^2_\sigma(G_m).
\]
For an open set $G$ and a function $f$ on $G$, we write $\mathtt e_G f$ for the zero extension of $f$ to $\R^3$. In particular, if $h\in L^2_\sigma(Y)$, then $\mathtt e_Y h$ is divergence-free in the sense of distributions on $\R^3$.

The linearized MHD operator on $L^2_\sigma(D)\times \mathscr X^2_\sigma(D)$ is
\begin{equation}
    \mathbb{S}
    \begin{pmatrix}
    u\\
    B
    \end{pmatrix}
    =
    \begin{pmatrix}
    -\mathbb{P}\Delta u - \mathbb{P}(H_*\cdot\nabla B + B\cdot\nabla H_*)\\
    \oprot(\oprot B) + \oprot(H_*\times u)
    \end{pmatrix},
\end{equation}
with domain
\begin{equation}
\begin{aligned}
    D(\mathbb{S})
    =
    &\big(H^2(D)^3\cap H^1_0(D)^3\cap L^2_\sigma(D)\big)\\
    &\times
    \Big\{
    B\in H^2(D)^3\cap \mathscr X^2_\sigma(D) 
    ~\Big|~
    B\cdot n|_{\partial D}=0,\ 
    \oprot B\times n|_{\partial D}=0
    \Big\}.
\end{aligned}
\end{equation}
The generation of the analytic semigroup for the linearized MHD system was classically studied by Yoshida-Giga \cite{YoshidaGiga1983}, and recently extended to general domains in works such as Monniaux \cite{Monniaux2021} and Kozono-Shimizu-Yanagisawa \cite{KozonoShimizuYanagisawa2025}.

The non-triviality of $X_{{\rm har}}(D)$ breaks the invertibility required for the resolvent analysis. To isolate this kernel part, we follow the functional framework of Kozono-Shimizu-Yanagisawa \cite{KozonoShimizuYanagisawa2025}: for a local solution $V=(v,H)$ of \eqref{eq.MHD} on $D$, we write
\begin{equation}
\begin{aligned}
    V(t)
    &=
    V_{{\rm evol}}(t) + V_{{\rm topo}},\\
    V_{{\rm evol}}(t)
    &\in
    L^2_\sigma(D)\times \mathscr X^2_\sigma(D),\\
    V_{{\rm topo}}
    &=
    (0,H_*')\in \{0\}\times X_{{\rm har}}(D).
\end{aligned}
\end{equation}
The two parts are treated by distinct approximation mechanisms. The evolutionary part $V_{{\rm evol}}(t)$ is treated by quantitative Runge approximations for the resolvent problem through the Dunford integral. Indeed, on the subspace $L^2_\sigma(D)\times \mathscr X^2_\sigma(D)$, the linearized operator leaves the magnetic term invariant; see \cite[Lemma 2.9]{KozonoShimizuYanagisawa2025}. Hence, provided that $H_*$ is sufficiently small, \cite[Proposition 1.1]{KozonoShimizuYanagisawa2025} yields the Dunford integral representation
\begin{equation}
    V_{{\rm evol}}(t)
    =
    \frac{1}{2\pi \ii}
    \int_\gamma 
    \exp(\lambda t)(\lambda+\mathbb{S})^{-1}V_0 
    \dd \lambda, 
\end{equation}
for initial data $V_0\in D(\mathbb{S})$ and a suitable contour $\gamma$ in the resolvent set. In contrast, the topological part $V_{{\rm topo}}$ is approximated directly by a stationary whole-space solution driven by time-independent remote forcings without using the Dunford integral. Theorem \ref{thm.main} then follows by superposing those two non-stationary and stationary approximations.

Let $H_*\in X_{{\rm har}}(D)$ satisfy the smallness assumption in Remark \ref{rem.smallness} below, and let $\tilde{H}_*$ be the extension given by Lemma \ref{lem.ext.H}. Fix a bounded open control set with smooth boundary
\begin{equation}\label{def.Y}
    Y\Subset \R^3\setminus \overbar{D},
\end{equation}
and let $B'$ be a ball sufficiently large so that $\overbar{D}\cup \overbar Y\subset B'$. We assume that every connected component of $B'\setminus \overbar{D}$ contains a nonempty open subset of $Y$.

We are now in a position to state the main result. Throughout this paper, we let $C>0$ denote a generic constant whose value may change from line to line. Any dependence of $C$ on non-fixed parameters will be indicated explicitly if necessary.

\begin{theorem}\label{thm.main}
For any $\rho_0\in(0,1)$, there exist positive constants $C,\nu,\eps_0$ such that the following statement holds: for every $0<\eps<\eps_0$ and every solution $V=(v,H)$ to \eqref{eq.MHD} that is represented as 
\begin{equation}\label{eq.V.thm.main}
    V(t) = V_{{\rm evol}}(t) + V_{{\rm topo}},
    \qquad
    V_{{\rm evol}}(t) = e^{-t {\mathbb{S}}} V_0,
    \qquad
    V_{{\rm topo}} = (0, H_*'), 
\end{equation}
for $V_0\in D({\mathbb{S}})$ and $H_*^\prime\in X_{{\rm har}}(D)$, there exist 
\begin{itemize}
\item
a pair of non-stationary forcings $(f,g) \in C^\infty([0,\infty);L^2(Y)^3\times L^2_\sigma(Y))$;

\item
a pair of stationary forcings $(i,j)\in L^2(Y)^3\times L^2_\sigma(Y)$; 
\end{itemize}
and a global approximation $U = (u,B)$ solving
\begin{equation}\label{eq.thm.main}
    \left\{
    \begin{array}{ll}
    \partial_{t} u - \Delta u 
    - \tilde{H}_*\cdot\nabla B - B\cdot\nabla \tilde{H}_* 
    + \nabla p 
    = \mathtt{e}_Y f
    + \mathtt{e}_Y i&\mbox{in}\ \R^3\times (0,\infty), \\
    \partial_{t} B + \oprot (\oprot B) + \oprot (\tilde{H}_* \times u)
    = \mathtt{e}_Y g + \mathtt{e}_Y j &\mbox{in}\ \R^3\times (0,\infty), \\
    \opdiv u 
    = \opdiv B 
    = 0&\mbox{in}\ \R^3\times [0,\infty), 
    \end{array}\right. 
\end{equation}
for some pressure $p$, such that 
\begin{equation}\label{est.thm.main}
    \|V(t) - U(t)\|_{L^2(D)}
    \le
    C\eps
    \|{\mathbb{S}} V_0\|_{L^2(D)}
    \exp(\rho_0 t)
    + C\eps
    \|H_*^\prime\|_{L^2(D)}, 
    \quad
    t \ge 0. 
\end{equation}
Moreover, $(f,g)$ is quantitatively estimated as
\begin{equation}\label{est.fg.thm.main}
    \|(f,g)(t)\|_{L^2(Y)} 
    \le
    \exp\big(C\exp(\eps^{-\nu})\big)
    \|\mathbb{S} V_0\|_{L^2(D)}
    \exp(\rho_0 t),
    \quad
    t \ge 0, 
\end{equation}
while $(i,j)$ is 
\begin{equation}\label{est.ij.thm.main}
    \|(i,j)\|_{L^2(Y)} 
    \le
    \exp(C\eps^{-\nu})
    \|H_*^\prime\|_{L^2(D)}. 
\end{equation}
\end{theorem}

\begin{remark}\label{rem1.thm.main}
\begin{enumerate}[(i)]
\item
The presence of a background magnetic field $H_*$ makes the linearized MHD system a system with variable coefficients. In the works \cite{HigakiSueur2025,Higaki2025_Runge}, the global approximation was constructed explicitly by using series expansions in (vector) spherical harmonics and by solving a system of ordinary differential equations (ODEs). While powerful, this ODE method relies heavily on the algebraic structure of constant coefficient operators, or separation of variables. It is inherently not robust under perturbations by variable coefficients; the loss of symmetry breaks the diagonal structure of the system, making explicit construction practically impossible for general $H_*$.

To overcome this limitation, we abandon constructing a global approximation that solves the homogeneous system. Instead, we adopt the ``source term approach'' inspired by R\"{u}land-Salo \cite{RülandSalo2019}. We prove the existence of a global solution under forcing supported in a control set $Y$. Then, by the quantitative unique continuation property rather than explicit calculation, the approach in \cite{RülandSalo2019} offers a robust way to estimate the cost of approximation, valid even for variable coefficient operators.

\item
In the context of MHD, Runge approximation on compact subsets is topologically trivial. Indeed, considering the approximations on an interior set $K$ of $D$ as in \cite{EncisoGarcía-FerreroPeralta-Salas2019,HigakiSueur2025} obscures the cohomological structure of the magnetic field. Locally, a harmonic magnetic field $H_\ast$ can be expressed as a gradient of a potential, losing its global circulatory character. The richness of the MHD system in a non-simply connected domain lies precisely in the interaction between the fluid and the background magnetic field $H_\ast$.

\item
The factors $\exp(\rho_0 t)$ in \eqref{est.thm.main} and \eqref{est.fg.thm.main} arise from the contour integral method used in the proof in Section \ref{sec.prf}; they are not claimed to reflect the sharp long-time behavior. The contour admissible simultaneously for $\mathbb{S}$ and $\mathbb{S}_{\R^3}$, the extended whole-space operator in Section \ref{sec.ext}, contains a low-frequency circular arc on which $\Re\lambda$ may be positive, while the available resolvent information does not allow us to replace it by a common contour contained in $\{\Re\lambda\le -c\}$ for some $c>0$. Accordingly, an estimate with a factor $\exp(-ct)$ for the evolutionary error would require additional low-frequency information or a different construction. Any such improvement is concerned only with the evolutionary part, since the topological part is time independent.
\end{enumerate}
\end{remark}

\begin{remark}\label{rem2.thm.main}
It is instructive to clarify the relationship between topology, forcing, and boundary conditions in the context of Runge-type global approximation.
\begin{itemize}
\item
{\bf Topology and Forcing}: For a bounded domain $D$, the condition $\beta_2(D)=0$ means that $\R^3\setminus \overbar{D}$ has no bounded connected components. This was the topological assumption used in \cite{Higaki2025_Runge} to construct a homogeneous global approximation. By allowing for forcing (source terms) supported in the complement $\R^3\setminus \overbar{D}$, as we do in this paper following \cite{RülandSalo2019}, the Runge approximation becomes feasible even when $\beta_2(D)>0$, akin to allowing poles in classical Walsh's theorem \cite{Walsh1929} for harmonic functions.

\item
{\bf Boundary Conditions}: The emergence of a non-trivial equilibrium depends critically on the boundary conditions. For the velocity field $v$, the no-slip condition precludes the existence of non-trivial harmonic fields, meaning the velocity part of the local solution effectively has no topological part. In contrast, for the magnetic field $H$, the perfect conductor condition $H \cdot n = 0$ allows for non-trivial harmonic fields. Remarkably, the coupling of the MHD system dictates the necessity of the stationary velocity forcing $i$ to maintain the equilibrium. As a byproduct, when $H_*=0$, our result covers the Stokes system, extending the scope of \cite{Higaki2025_Runge}.
\end{itemize}
\end{remark}

The rest of this paper is organized as follows. Section \ref{sec.ext} gives an extension of the MHD operator $\mathbb{S}$ to the whole space $\R^3$. Section \ref{sec.resol} presents the quantitative Runge approximation for the resolvent problem, which handles the time-evolving part. Section \ref{sec.harm} shows the global approximation of harmonic vector fields, representing the stationary topological part. Finally, Section \ref{sec.prf} is devoted to the proof of Theorem \ref{thm.main}, synthesizing these approximations with the frequency decomposition of the resolvent parameter.

    \section{Extension and adjoint of the MHD operator}
    \label{sec.ext}

For the use in the following sections, we need to extend the MHD operator $\mathbb{S}$ on $D$ to the one on $\R^3$. For this purpose, we extend $H_*\in X_{{\rm har}}(D)$, which defines $\mathbb{S}$, to $\R^3$ by cut-off and keeping divergence-free as follows. Let $R>0$ be sufficiently large so that $\overbar{D}\subset B_R$.

\begin{lemma}\label{lem.ext.H}
There exists an extension $\tilde{H}_* \in W^{1,\infty}(\R^3)^3$ such that 
\[
    \tilde{H}_*|_{D} = H_*,
    \qquad
    \opdiv \tilde{H}_* = 0,
    \qquad
    \opspt \tilde{H}_* \Subset B_R.
\]
Moreover,
\[
    \|\tilde{H}_*\|_{W^{1,\infty}(\R^3)}
    \le
    C\|H_*\|_{W^{1,\infty}(D)}.
\]
The constant $C>0$ is independent of $H_*$. 
\end{lemma}

\begin{proof}
Let $\{h_1,\dots,h_m\}$ be a basis of $X_{{\rm har}}(D)$. Since $h_j\cdot n=0$ on $\partial D$, we have
\begin{equation}\label{eq1.prf.lem.ext.H}
    \int_{\Gamma_\ell} 
    h_j\cdot n
    \dd \sigma=0, 
\end{equation}
for every connected component $\Gamma_\ell$ of $\partial D$.

Fix $p>3$. Thanks to the divergence-free extension theorem of Kato-Mitrea-Ponce-Taylor \cite[Corollary 3.2]{KatoMitreaPonceTaylor2000} with $s=2$, there exists $\tilde{h}_j\in W^{2,p}(\R^3)^3$ such that
\[
    \tilde{h}_j|_D=h_j,
    \qquad
    \opdiv \tilde{h}_j=0,
    \qquad
    \opspt \tilde{h}_j \Subset B_R.
\]
Note that the flux coefficients $\lambda_j$ appearing in \cite[Corollary 3.2]{KatoMitreaPonceTaylor2000} vanish due to \eqref{eq1.prf.lem.ext.H}. Since $p>3$, the Sobolev embedding gives
$
    W^{2,p}(\R^3)
    \hookrightarrow
    W^{1,\infty}(\R^3)
$.

Now write
\[
    H_* 
    = 
    \sum_{j=1}^m c_j h_j,
    \qquad
    \tilde{H}_* 
    := 
    \sum_{j=1}^m c_j \tilde{h}_j.
\]
Then the extension properties required in the statement follow by linearity. Moreover, since $X_{{\rm har}}(D)$ is finite-dimensional, all norms are equivalent on it, and therefore
\[
    \|\tilde{H}_*\|_{W^{1,\infty}(\R^3)}
    \le
    C \sum_{j=1}^m |c_j|
    \le
    C \|H_*\|_{W^{1,\infty}(D)}.
\]
This completes the proof. 
\end{proof}

Let us denote by $\mathbb{P}_{\R^3}$ the Helmholtz projection in $\R^3$, by $\mathbb{A}_{\R^3}:=-\mathbb{P}_{\R^3}\Delta$ the Stokes operator on $L^2_\sigma(\R^3)$, and by $\mathbb{L}_{\R^3}:=\oprot\oprot$ the magnetic operator on $L^2_\sigma(\R^3)$.

Using the extension $\tilde{H}_*$ in Lemma \ref{lem.ext.H}, we define 
\[
\begin{split}
    \mathbb{J}_{\R^3} u 
    &= \oprot (\tilde{H}_* \times u), \\
    \mathbb{K}_{\R^3} B
    &= -\mathbb{P}_{\R^3}
    (\tilde{H}_* \cdot\nabla B 
    + B\cdot\nabla \tilde{H}_*). 
\end{split}
\]
Then the global linearized MHD operator $\mathbb{S}_{\R^{3}}$ is defined by
\[
    \mathbb{S}_{\R^3}
    \begin{pmatrix}
    u\\
    B
    \end{pmatrix}
    =
    \begin{pmatrix}
    \mathbb{A}_{\R^3} & \mathbb{K}_{\R^3}\\
    \mathbb{J}_{\R^3} & \mathbb{L}_{\R^3}
    \end{pmatrix}
    \begin{pmatrix}
    u\\
    B
    \end{pmatrix}, 
    \qquad
    D(\mathbb{S}_{\R^3})
    = 
    D(\mathbb{A}_{\R^3})\times D(\mathbb{L}_{\R^3}). 
\]
The adjoint operator $\mathbb{S}_{\R^{3}}^{*}$ is given as follows.

\begin{lemma} \label{lem.adjoint}
The adjoint operator $\mathbb{S}_{\R^{3}}^{*}$ is given by 
\[
    \mathbb{S}_{\R^3}^*
    \begin{pmatrix}
    u\\
    B
    \end{pmatrix}
    =
    \begin{pmatrix}
    \mathbb{A}_{\R^3} & \mathbb{J}_{\R^3}^* \\
    \mathbb{K}_{\R^3}^* & \mathbb{L}_{\R^3}
    \end{pmatrix}
    \begin{pmatrix}
    u\\
    B
    \end{pmatrix}, 
    \qquad
    D(\mathbb{S}_{\R^3}^{*})
    = 
    D(\mathbb{S}_{\R^3}), 
\]
where the adjoint coupling operators are defined by
\[
\begin{split}
    \mathbb{J}_{\R^3}^* B 
    &= \mathbb{P}_{\R^3} ((\oprot B) \times \tilde{H}_*), \\
    \mathbb{K}_{\R^3}^* u 
    &= \mathbb{P}_{\R^3}
    (\tilde{H}_* \cdot\nabla u 
    - (\nabla \tilde{H}_*)^\top u). 
\end{split}
\]
\end{lemma}

\begin{proof}
Note that $\mathbb{A}_{\R^3}$ and $\mathbb{L}_{\R^3}$ are self-adjoint on their respective domains.

For $\mathbb{J}_{\R^3}^*$, by the identity $(A \times B) \cdot C = A \cdot (B \times C)$ and integration by parts, 
\[
\begin{split}
    \langle \mathbb{J}_{\R^3} u, B \rangle 
    &= \langle \oprot (\tilde{H}_* \times u), B \rangle \\
    &= \langle \tilde{H}_* \times u, \oprot B \rangle \\
    &= \langle u, \oprot B \times \tilde{H}_* \rangle.
\end{split}
\]
Applying the projection $\mathbb{P}_{\R^3}$ leads to the formula of $\mathbb{J}_{\R^3}^*$.

For $\mathbb{K}_{\R^3}^*$, by the definition, 
\[
\begin{split}
    \langle \mathbb{K}_{\R^3} B, u \rangle 
    &= \langle -\mathbb{P}_{\R^3}(\tilde{H}_*\cdot\nabla B+B\cdot\nabla\tilde{H}_*), u \rangle \\
    &= -\langle \tilde{H}_*\cdot\nabla B, u \rangle
    - \langle B\cdot\nabla\tilde{H}_*, u \rangle.
\end{split}
\]
For the first term, since $\opdiv \tilde{H}_* = 0$, 
\[
    \langle -\tilde{H}_* \cdot \nabla B, u \rangle = \langle B, \tilde{H}_* \cdot \nabla u \rangle.
\]
For the second term, we look at the $i$-th component $-B_k \partial_k (\tilde{H}_*)_i$: 
\[
    \int_{\R^3} 
    \big(-B_k \partial_k (\tilde{H}_*)_i\big) 
    \overbar{u}_i \dd x 
    = 
    \int_{\R^3} 
    B_k 
    \big(-\partial_k (\tilde{H}_*)_i \overbar{u}_i\big) \dd x.
\]
The term in the parenthesis is the $k$-th component of $-(\nabla \tilde{H}_*)^\top u$. Combining these and applying the projection $\mathbb{P}_{\R^3}$ leads to the formula of $\mathbb{K}_{\R^3}^*$. This completes the proof. 
\end{proof}

\begin{remark}\label{rem.smallness}
Let $C_{\rm ext}$ denote the operator norm of the extension mapping in Lemma \ref{lem.ext.H}. For the operator $\mathbb{S}$ on the bounded domain $D$, Kozono-Shimizu-Yanagisawa \cite[Proposition 1.1]{KozonoShimizuYanagisawa2025} gives $\delta_{\rm bd}>0$ such that $\mathbb{S}$ is sectorial on $L^2_\sigma(D)\times \mathscr X^2_\sigma(D)$ whenever
\[
    \|H_*\|_{L^\infty(D)}
    + \|\nabla H_*\|_{L^\infty(D)}
    \le
    \delta_{\rm bd}.
\]
For the whole-space operator, we write
\[
    \mathbb{S}^{(0)}_{\R^3}
    :=
    \begin{pmatrix}
    \mathbb{A}_{\R^3}&0\\
    0&\mathbb{L}_{\R^3}
    \end{pmatrix},
    \qquad
    \mathbb{S}^{(1)}_{\R^3}
    :=
    \begin{pmatrix}
    0&\mathbb{K}_{\R^3}\\
    \mathbb{J}_{\R^3}&0
    \end{pmatrix}.
\]
Since $\mathbb{A}_{\R^3}$ and $\mathbb{L}_{\R^3}$ are positive self-adjoint, $\mathbb{S}^{(0)}_{\R^3}$ is sectorial on $L^2_\sigma(\R^3)\times L^2_\sigma(\R^3)$. Moreover, standard interpolation gives, for every $\eta>0$ and every $U$ in $D(\mathbb{S}^{(0)}_{\R^3})$,
\[
    \|\mathbb{S}^{(1)}_{\R^3}U\|_{L^2(\R^3)}
    \le
    C\|\tilde{H}_*\|_{W^{1,\infty}(\R^3)}
    \big(
    \eta \|\mathbb{S}^{(0)}_{\R^3}U\|_{L^2(\R^3)}
    + C_\eta \|U\|_{L^2(\R^3)}
    \big).
\]
Hence there exists $\delta_{\rm ws}>0$ such that $\mathbb{S}_{\R^3}=\mathbb{S}^{(0)}_{\R^3}+\mathbb{S}^{(1)}_{\R^3}$ and likewise $\mathbb{S}_{\R^3}^*$ are sectorial whenever $\|\tilde{H}_*\|_{W^{1,\infty}(\R^3)}\le \delta_{\rm ws}$, by the perturbation theorem for sectorial operators; see \cite[Theorem 1.2.14]{Lunardi1995}. Finally, let $\delta_{\rm coe}>0$ be a small constant ensuring the coercivity in the proof of Lemma \ref{lem.solver.topo}. We therefore fix once and for all
\[
    \delta_*
    :=
    \min\{
    \delta_{\rm bd},\, 
    C_{\rm ext}^{-1}\delta_{\rm ws},\, 
    C_{\rm ext}^{-1}\delta_{\rm coe}
    \}
\]
and assume that
\[
    \|H_*\|_{W^{1,\infty}(D)}
    \le
    \delta_*.
\]
This assumption will be used throughout whenever the smallness of $H_*$ is required.
\end{remark}

    \section{Global approximation of resolvent solutions}
    \label{sec.resol}

In this section, we present the quantitative Runge approximation for the resolvent problem associated with the linearized MHD operator $\mathbb{S}$. This analysis serves as the building block for approximating the time-evolving part $V_{{\rm evol}}$ introduced in \eqref{eq.V.thm.main} in the frequency domain.

For $\del\in(0,\pi)$, we set
\[
    \Sigma_\del
    :=
    \{z\in\C\setminus\{0\} \mid |\oparg z|<\del\}.
\]
Throughout this section, let
\begin{equation}\label{assump.lambda}
    \lambda
    \in
    \Sigma_{\pi-\del} \cap \{|\lambda| \ge \rho_0\},
\end{equation}
with $\del\in(0,\pi/2)$ and $\rho_0>0$. We then consider $V=(v,H)$ solving
\begin{equation}\label{eq.MHD.resol.loc}
    \left\{
    \begin{array}{ll}
    \lambda v - \Delta v 
    - H_*\cdot\nabla H - H\cdot\nabla H_* 
    + \nabla q
    = 0&\mbox{in}\ D, \\
    \lambda H + \oprot (\oprot H) + \oprot (H_* \times v)
    = 0&\mbox{in}\ D, \\
    \opdiv v
    = \opdiv H 
    = 0&\mbox{in}\ D.
    \end{array}\right.
\end{equation}

We start with the interpolation estimate for the vorticity.

\begin{lemma}\label{lem.vor.interp}
Let $Y\subset \R^3$ be a bounded Lipschitz domain, and $Y_0\Subset Y$ be a bounded Lipschitz subdomain. Then there exists $C>0$ such that, for every $\Psi\in H^2(Y)^3$,
\begin{equation}\label{est.lem.vor.interp}
    \|\oprot\Psi\|_{L^2(Y_0)}
    \le
    C
    \|\mathbb{P}_Y \Psi\|_{L^2(Y)}^{1/2}
    \|\Psi\|_{H^2(Y)}^{1/2}.
\end{equation}
\end{lemma}

\begin{proof}
Set $b=\mathbb{P}_Y \Psi$. By the Helmholtz decomposition in $Y$, there exists $\chi\in H^1(Y)$ such that $\Psi = b + \nabla \chi$. Thus $\oprot\Psi = \oprot b$ in $Y$ in the sense of distributions.

Let $\zeta\in H^1_0(Y)^3$. Then we have 
\[
\begin{split}
    \Big|
    \langle 
    \oprot\Psi,
    \zeta
    \rangle_{H^{-1}(Y),H^1_0(Y)}
    \Big|
    &=
    \bigg|
    \int_Y b\cdot \oprot \zeta \dd x
    \bigg| \\
    &\le
    C
    \|b\|_{L^2(Y)}
    \|\zeta\|_{H^1(Y)}.
\end{split}
\]
Hence 
$
    \|\oprot\Psi\|_{H^{-1}(Y)}
    \le 
    C\|\mathbb{P}_Y\Psi\|_{L^2(Y)}
$.
If $\zeta\in H^1_0(Y_0)^3$, its zero
extension $\widetilde{\zeta}$ to $Y$ satisfies $\widetilde{\zeta}\in H^1_0(Y)^3$ and
$\|\widetilde{\zeta}\|_{H^1(Y)}=\|\zeta\|_{H^1(Y_0)}$. This implies that 
\begin{equation}\label{est1.prf.lem.vor.interp}
    \|\oprot\Psi\|_{H^{-1}(Y_0)}
    \le
    C
    \|\mathbb{P}_Y\Psi\|_{L^2(Y)}.
\end{equation}
On the other hand, since $\Psi\in H^2(Y)^3$,
\begin{equation}\label{est2.prf.lem.vor.interp}
    \|\oprot\Psi\|_{H^1(Y_0)}
    \le
    C
    \|\Psi\|_{H^2(Y)}.
\end{equation}
The standard interpolation inequality on $Y_0$ gives
\begin{equation}\label{est3.prf.lem.vor.interp}
    \|\oprot\Psi\|_{L^2(Y_0)}
    \le
    C
    \|\oprot\Psi\|_{H^{-1}(Y_0)}^{1/2}
    \|\oprot\Psi\|_{H^1(Y_0)}^{1/2}.
\end{equation}
The desired estimate \eqref{est.lem.vor.interp} follows from \eqref{est1.prf.lem.vor.interp}--\eqref{est3.prf.lem.vor.interp}.
\end{proof}

Next we state the stability estimate for the system \eqref{eq.lem.stab} below.

\begin{lemma}\label{lem.stab}
Let $G\subset\R^3$ be a bounded domain with smooth boundary, $G'$ be a subdomain of $G$ satisfying $\overbar{G'} \subset G$. Let $(\varphi,\omega)\in H^2_{{\rm loc}}(G)^6 \cap H^1(G)^6$ satisfy, for some $s\in H^1(G)$, 
\begin{equation}\label{eq.lem.stab}
    \left\{
    \begin{array}{ll}
    \lambda \varphi
    - \Delta \varphi
    + \omega \times \tilde{H}_*
    + \nabla s
    = 0
    & \mbox{in}\ G,\\
    \lambda \omega
    - \Delta \omega
    + \oprot
    (
    \tilde{H}_* \cdot\nabla \varphi 
    - (\nabla \tilde{H}_*)^\top \varphi
    )
    = 0
    & \mbox{in}\ G,\\
    \opdiv \varphi
    =
    \opdiv \omega
    = 0
    & \mbox{in}\ G.
    \end{array}
    \right.
\end{equation}
Suppose that, for $0<\eta<\mathcal{E}$, we have $\|(\varphi,\omega)\|_{H^1(G)} \le \mathcal{E}$ and $\|(\varphi,\omega)\|_{L^2(G')} \le \eta$. Then there exist $C,\mu>0$ independent of $\lambda, \eta, \mathcal{E}$ such that 
\begin{equation}\label{est.lem.stab}
    \|(\varphi,\omega)\|_{L^2(G)}
    \le 
    C 
    \exp
    (
    C \langle \lambda \rangle
    )\,
    \mathcal{E} 
    \bigg(\log \frac{\mathcal{E}}{\eta} \bigg)^{-\mu}. 
\end{equation}
In particular, if $(\varphi,\omega)=0$ in a nonempty open subset of $G$, then $(\varphi,\omega)=0$ in $G$.
\end{lemma}

\begin{proof}
Set
\[
    (\Phi(x,t),\Omega(x,t))
    =
    \exp(t\lambda) (\varphi(x),\omega(x)),
    \qquad
    S(x,t)
    = 
    \exp(t\lambda) s(x)
\]
on $G\times(-1,1)$. Then we see that
\begin{equation}\label{eq1.prf.lem.stab}
    \left\{
    \begin{array}{ll}
    \partial_t \Phi
    - \Delta \Phi
    + \Omega \times \tilde{H}_*
    + \nabla S
    = 0
    & \mbox{in}\ G\times(-1,1),\\
    \partial_t \Omega
    - \Delta \Omega
    + \oprot
    (
    \tilde{H}_* \cdot\nabla \Phi
    - (\nabla \tilde{H}_*)^\top \Phi
    )
    = 0
    & \mbox{in}\ G\times(-1,1),\\
    \opdiv \Phi
    =
    \opdiv \Omega
    = 0
    & \mbox{in}\ G\times(-1,1).
    \end{array}
    \right.    
\end{equation}

We first derive a $3$-cylinder estimate for \eqref{eq1.prf.lem.stab}. Set $Q=\oprot \Phi$. Since $\opdiv\Phi=0$, 
\begin{equation}\label{eq2.prf.lem.stab}
    \Delta \Phi 
    + \oprot Q 
    = 0.
\end{equation}
Let $\mathcal{P}(a)$ denote the skew-symmetric matrix satisfying
$\opDiv\mathcal{P}(a)=\oprot a$ for a vector field $a$. Then, taking the rotation of the first equation of \eqref{eq1.prf.lem.stab} gives
\begin{equation}\label{eq3.prf.lem.stab}
    \left\{
    \begin{array}{ll}
    \partial_t Q
    - \Delta Q
    + \opDiv F
    = 0
    & \mbox{in}\ G\times(-1,1),\\
    \partial_t \Omega
    - \Delta \Omega
    + \opDiv G
    = 0
    & \mbox{in}\ G\times(-1,1),
    \end{array}
    \right.    
\end{equation}
where
\[
    F
    :=
    \mathcal{P}(\Omega\times\tilde{H}_*),
    \qquad
    G
    :=
    \mathcal{P}
    \big(
    \tilde{H}_*\cdot\nabla\Phi
    -(\nabla\tilde{H}_*)^\top\Phi
    \big).
\]
Since $\tilde{H}_*\in W^{1,\infty}(\R^3)^3$, we have 
\[
    |F| + |G|
    \le
    C
    (
    |\Omega|+|\nabla\Phi|+|\Phi|
    ).
\]
The system \eqref{eq2.prf.lem.stab}--\eqref{eq3.prf.lem.stab} has the same diagonal principal part as the generalized non-stationary Stokes system in \cite{LinWang2022}, while the remaining terms are bounded first- and zero-order perturbations allowed in the argument of \cite{LinWang2022}. Hence, applying the parabolic and elliptic Carleman estimates yields the following three-cylinder inequality: for $0<t_0<T<1$ and concentric balls $B_{R_1}\subset B_{R_2}\subset B_{R_3}\subset G$ under radius conditions,
\begin{equation}\label{est1.prf.lem.stab}
\begin{split}
    &\int_{-T+t_0}^{T-t_0}
    \int_{B_{R_2}}
    \big(
    |\Phi|^2+|\Omega|^2
    \big)
    \dd x\dd t \\
    &\le
    C
    \left(
    \int_{-T}^{T}
    \int_{B_{R_1}}
    \big(
    |\Phi|^2+|\Omega|^2
    \big)
    \dd x\dd t
    \right)^{\theta}
    \left(
    \int_{-T}^{T}
    \int_{B_{R_3}}
    \big(
    |\Phi|^2+|\Omega|^2
    \big)
    \dd x\dd t
    \right)^{1-\theta},
\end{split}
\end{equation}
where $\theta\in(0,1)$ depends on $G,t_0,T,R_1,R_2,R_3$ and $\|\tilde{H}_*\|_{W^{1,\infty}}$.

Returning to the time-independent field $(\varphi,\omega)$ by using 
\[
    \int_{-a}^{a}
    \int_{B_R}
    (|\Phi|^2+|\Omega|^2)
    \dd x\dd t
    =
    \bigg(
    \int_{-a}^{a}
    \exp(2t \Re\lambda)
    \dd t
    \bigg)
    \|(\varphi,\omega)\|_{L^2(B_R)}^2
\]
for $a\in(0,1)$ and ball $B_R\subset G$, gives the $3$-sphere estimate
\[
    \|(\varphi,\omega)\|_{L^2(B_{R_2})}
    \le
    C
    \exp
    (
    C \langle \lambda \rangle
    )\,
    \|(\varphi,\omega)\|_{L^2(B_{R_1})}^{\theta}
    \|(\varphi,\omega)\|_{L^2(B_{R_3})}^{1-\theta},
\]
for concentric balls $B_{R_1}\subset B_{R_2}\subset B_{R_3}\subset G$. The standard propagation-of-smallness argument of \cite{AlessandriniRondiRossetVessella2009} then yields \eqref{est.lem.stab}. The qualitative unique continuation follows by applying the estimate with $\eta\downarrow0$ along a chain of balls from the open set where $(\varphi,\omega)$ vanishes.
\end{proof}

We now proceed with the approximation of $V=(v,H)$ by a global solution under forcing supported in $Y$ in \eqref{def.Y}. Let $\mathcal{X}_{\rm restr}$ denote the subspace of $L^2(D)^6$
\begin{equation}\label{def.Xrestr}
    \mathcal{X}_{\rm restr}
    = 
    \overline{
    \{a|_D
    \mid
    a\in L^2_\sigma(\R^3)
    \}}^{L^2(D)^3}
    \times
    \overline{
    \{a|_D
    \mid
    a\in L^2_\sigma(\R^3)
    \}}^{L^2(D)^3}.
\end{equation}
Then we define $\mathcal{X}^{\rm tgt}$ as the space of target functions by 
\begin{equation}\label{def.Xtarget}
    \mathcal{X}^{\rm tgt}
    = 
    \Big\{V\in \mathcal{X}_{\rm restr}\cap H^2(D)^6
    ~\big|~ 
    \mbox{
    $V=(v,H)$ solves \eqref{eq.MHD.resol.loc} with $q\in H^1(D)$
    }
    \Big\}.
\end{equation}
Note that the normal trace of $(v,H)\in \mathcal{X}^{\rm tgt}$ satisfies 
\begin{equation}\label{eq.flux}
    \langle \gamma_n v,1\rangle_{\Gamma_\ell}
    =
    \langle \gamma_n H,1\rangle_{\Gamma_\ell}
    =
    0
\end{equation}
for every connected component $\Gamma_\ell$ of $\partial D$.

We also define the linear operator 
\[
    T: L^2(Y)^3 \times L^2_\sigma(Y) \to L^2(D)^6 
\]
as follows: for $F=(f,g)\in L^2(Y)^3 \times L^2_\sigma(Y)$, 
\begin{equation}\label{def.T.J}
    TF 
    = 
    \big(
    (\lambda + {\mathbb{S}}_{\R^3})^{-1} 
    J_Y F
    \big)\big|_D,
    \qquad
    J_Y F
    :=
    (
    \mathbb{P}_{\R^3} \mathtt{e}_Y f, 
    \mathtt{e}_Y g
    ). 
\end{equation}
Moreover, 
\begin{equation}\label{def.X}
    \mathcal{X}
    =
    \overbar{\opRan(T)}^{L^2(D)^6}.
\end{equation}
Let us denote by $\mathcal{K}$ the kernel of $T$ and by $\mathcal{K}^\perp$ the orthogonal space of $\mathcal{K}$ in $L^2(Y)^3 \times L^2_\sigma(Y)$. Finally, let us denote by $A$ the restriction of $T$ to $\mathcal{K}^\perp$.

\begin{lemma}\label{lem.A}
The following hold. 
\begin{enumerate}[(1)]
\item\label{item1.lem.A}
For $V=(v,H)\in \mathcal{X}$, set
\begin{equation}\label{def.Pi}
    (\varphi,\Psi)
    = 
    (\overbar{\lambda} + \mathbb{S}_{\R^3}^*)^{-1}
    \Pi_D V, 
    \qquad
    \Pi_D V
    := 
    (
    \mathbb{P}_{\R^3} \mathtt{e}_D v,    
    \mathbb{P}_{\R^3} \mathtt{e}_D H
    ). 
\end{equation}
Then the adjoint $A^*: \mathcal{X} \to \mathcal{K}^\perp$ of $A$ is given by 
\[
    A^* V 
    = 
    \big(
    \varphi|_{Y},
    \mathbb{P}_{Y}(\Psi|_{Y})
    \big).
\]

\item\label{item2.lem.A}
The operator $A^* A$ is positive, compact, and self-adjoint on $\mathcal{K}^\perp$.

\item\label{item3.lem.A}
The range of $A$ is dense in $\mathcal{X}$. Moreover, $\mathcal{X}^{\rm tgt}$ is contained in $\mathcal{X}$.

\item\label{item4.lem.A}
There exist orthonormal bases 
\[
    \{F_j\}_{j=1}^{\infty} \subset \mathcal{K}^\perp, 
    \qquad
    \{V_j\}_{j=1}^{\infty} \subset \mathcal{X} 
\]
and positive constants $\{\alpha_j\}_{j=1}^{\infty}$ with $\alpha_1\ge \alpha_2\ge\cdots>0$ and $\alpha_j\to0$ such that 
\[
    A F_j = \alpha_j V_j, 
    \qquad
    A^* V_j = \alpha_j F_j. 
\]
\end{enumerate}
\end{lemma}

\begin{proof}
We first prove (\ref{item1.lem.A}). Let $V=(v,H)\in \mathcal{X}$ and set
\[
    W
    =
    (\varphi,\Psi)
    =
    \big((\overbar{\lambda} + \mathbb{S}_{\R^3}^*)^{-1}\big)
    \Pi_D V.
\]
For every $F=(f,g)$ in $\mathcal{K}^\perp$, putting
\[
    U
    =
    (\lambda + {\mathbb{S}}_{\R^3})^{-1} 
    J_Y F, 
\]
we compute
\[
    \langle AF,V\rangle_D
    =
    \langle
    U|_D,
    V
    \rangle_D 
    =
    \langle
    U,
    \mathtt{e}_D V
    \rangle_{\R^3} 
    =
    \langle
    U,
    \Pi_D V
    \rangle_{\R^3} 
    =
    \langle J_Y F, W\rangle_{\R^3}.
\]
By the definition of $J_Y$, we see from $g\in L^2_\sigma(Y)$ that 
\[
\begin{split}
    \langle J_Y F, W\rangle_{\R^3}
    &=
    \langle
    \mathbb{P}_{\R^3} \mathtt{e}_Y f,
    \varphi
    \rangle_{\R^3}
    +
    \langle
    \mathtt{e}_Y g, 
    \Psi
    \rangle_{\R^3} \\
    &=
    \langle
    f,
    \varphi|_{Y}
    \rangle_{Y}
    +
    \langle
    g, 
    \mathbb{P}_{Y}(\Psi|_{Y})
    \rangle_{Y} \\
    &=
    \big\langle
    F,
    \big(
    \varphi|_{Y},
    \mathbb{P}_{Y}(\Psi|_{Y})
    \big)
    \big\rangle_{Y}.
\end{split}
\]
Moreover, if $G\in \mathcal{K}$, then $TG=0$, and the same computation gives
\[
    \big\langle
    G,
    \big(
    \varphi|_{Y},
    \mathbb{P}_{Y}(\Psi|_{Y})
    \big)
    \big\rangle_{Y}
    =
    \langle TG, V\rangle_D
    = 0.
\]
Thus 
$
    \big(
    \varphi|_{Y},
    \mathbb{P}_{Y}(\Psi|_{Y})
    \big)
$ 
belongs to $\mathcal{K}^\perp$, and the formula for $A^* V$ follows.

For (\ref{item2.lem.A}), the resolvent estimate gives
$
    \|TF\|_{H^2(D)}
    \le
    C
    \|F\|_{L^2(Y)}
$
for $\lambda$ in the fixed sector with $|\lambda| \ge \rho_0$ with $C$ depending on $\rho_0$. By the Rellich theorem, $A$ is compact from $\mathcal{K}^\perp$ into $\mathcal{X}\subset L^2(D)^6$. Consequently $A^* A$ is positive, compact, and self-adjoint on $\mathcal{K}^\perp$.

We now prove (\ref{item3.lem.A}). The density is due to the definition \eqref{def.X}. Next we show that $\mathcal{X}^{\rm tgt}\subset\mathcal{X}$. For $V=(v,H)\in L^2(D)^6$ orthogonal to $\mathcal{X}$, let $(\varphi,\Psi)$ be defined by \eqref{def.Pi}. Since the computation used in the proof of (\ref{item1.lem.A}) is valid for every
$F\in L^2(Y)^3\times L^2_\sigma(Y)$, 
\[
    0
    =
    \langle TF,V\rangle_D 
    =
    \big\langle
    F,
    \big(
    \varphi|_{Y},
    \mathbb{P}_{Y}(\Psi|_{Y})
    \big)
    \big\rangle_{Y}. 
\]
Consequently,
\[
    \varphi|_Y=0,
    \qquad
    \mathbb{P}_Y(\Psi|_Y)=0.
\]
Set 
\[
    \omega=\oprot\Psi. 
\]
Then the second equality implies that $\Psi|_Y$ is locally of gradient form, and hence $\omega|_Y=0$. Let $\mathcal{O}_0,\ldots,\mathcal{O}_M$ be the connected components of $B'\setminus\overbar D$, and choose nonempty subdomains $Y_k\Subset Y\cap\mathcal{O}_k$. Since the forcing $\mathtt{e}_D V$ vanishes in $\mathcal{O}_k$, the pair $(\varphi,\omega)$ satisfies \eqref{eq.lem.stab}. Lemma \ref{lem.stab} therefore gives $(\varphi,\omega)=0$ in $\mathcal{O}_k$ for every $k$, and thus 
\[
    \varphi|_{B'\setminus\overbar{D}}=0,
    \qquad
    \omega|_{B'\setminus\overbar{D}}=0.
\]

We claim that every $U\in \mathcal{X}^{\rm tgt}$ is orthogonal to the subspace $\mathcal{X}^{\perp}$. We may show that $\langle U,V\rangle_D=0$ holds for every $V\in\mathcal{X}^\perp$, and therefore conclude that $U\in(\mathcal{X}^\perp)^\perp=\mathcal{X}$. Let $U=(u,B)$ be associated with a pressure $p\in H^1(D)$, that is, $U=(u,B)$ satisfies
\begin{equation}\label{eq0.prf.lem.A}
\begin{split}
    \left\{
    \begin{array}{ll}
    \lambda u - \Delta u 
    - H_*\cdot\nabla B - B\cdot\nabla H_* 
    + \nabla p
    = 0&\mbox{in}\ D, \\
    \lambda B + \oprot (\oprot B) + \oprot (H_* \times u)
    = 0&\mbox{in}\ D, \\
    \opdiv u
    = \opdiv B 
    = 0&\mbox{in}\ D.
    \end{array}\right.
\end{split}
\end{equation}
On the other hand, there exist $s_1,s_2\in H^1(B')$ such that
\begin{equation}\label{eq1.prf.lem.A}
\begin{split}
    \left\{
    \begin{array}{ll}
    \overbar{\lambda} \varphi
    - \Delta\varphi
    + (\oprot\Psi)\times\tilde{H}_*
    + \nabla s_1
    = \mathtt{e}_D v &\mbox{in}\ B', \\
    \overbar{\lambda} \Psi
    + \oprot \omega
    + 
    \tilde{H}_*\cdot\nabla\varphi
    - (\nabla\tilde{H}_*)^\top\varphi
    + \nabla s_2
    = \mathtt{e}_D H &\mbox{in}\ B', \\
    \opdiv \varphi
    = \opdiv \Psi
    = 0&\mbox{in}\ B'.
    \end{array}\right. 
\end{split}
\end{equation}
Blockwise application of the Green formula to \eqref{eq0.prf.lem.A}--\eqref{eq1.prf.lem.A} gives
\begin{equation}\label{eq2.prf.lem.A}
    \langle U,V\rangle_D
    =
    \mathcal{B}_{{\rm Stokes}}(u,p,\varphi)
    + \mathcal{B}_{{\rm Mag}}(B,\Psi)
    + \mathcal{B}_{{\rm Coup}}(U,W)
    + \mathcal{B}_{{\rm pr},2}(B,s_2),
\end{equation}
where
\begin{equation}\label{def.B}
\begin{split}
    \mathcal{B}_{{\rm Stokes}}(u,p,\varphi)
    &:=
    \int_{\partial D}
    \Big(
    (\partial_n u - pn)\cdot \overbar{\varphi}
    + (u\cdot n) \overbar{s_1}
    - u\cdot \partial_n\overbar{\varphi}
    \Big)
    \dd \sigma,\\
    \mathcal{B}_{{\rm Mag}}(B,\Psi)
    &:=
    \int_{\partial D}
    \Big(
    (B\times n)\cdot \overbar{\omega}
    + (\oprot B\times n)\cdot \overbar{\Psi}
    \Big)
    \dd \sigma,\\
    \mathcal{B}_{{\rm Coup}}(U,W)
    &:=
    -\int_{\partial D}
    (u\cdot n)(H_*\cdot \overbar{\Psi})
    \dd \sigma,\\
    \mathcal{B}_{{\rm pr},2}(B,s_2)
    &:=
    \int_{\partial D}
    (B\cdot n)\,\overbar{s_2}\dd \sigma.
\end{split}
\end{equation}
The detailed derivation of \eqref{eq2.prf.lem.A} is given in Appendix \ref{appx.bdry.int}.

Because $(\varphi,\omega)=0$ in $B'\setminus \overbar{D}$ and $(\varphi,\Psi)$ is in $H^2(B')$, the trace of $(\varphi,\omega)$ and the first-order derivatives of $\varphi$ vanish on $\partial D$ when approached from the interior. Moreover, since $H_*\cdot n=0$ on $\partial D$, the vector identity $(a\times b)\times c=(a\cdot c)b-(b\cdot c)a$ gives 
\[
\begin{split}
    \oprot B\times n - (u\cdot n)H_*
    =
    (\oprot B + H_*\times u)\times n 
\end{split}
\]
on $\partial D$. Thus the equality \eqref{eq2.prf.lem.A} is reduced to 
\begin{equation}\label{eq3.prf.lem.A}
\begin{split}
    \langle U,V\rangle_D 
    &=
    \int_{\partial D}
    (u\cdot n) \overbar{s_1}
    \dd \sigma \\
    &\quad
    + 
    \int_{\partial D}
    ((\oprot B + H_*\times u)\times n)\cdot \overbar{\Psi}\dd \sigma 
    + 
    \int_{\partial D}
    (B\cdot n)\,\overbar{s_2}\dd \sigma. 
\end{split}
\end{equation}

Let us show that the right-hand side is zero. In $B'\setminus\overbar{D}$, we see from \eqref{eq1.prf.lem.A} that 
\begin{equation}\label{eq4.prf.lem.A}
    \nabla s_1=0,
    \qquad
    \overbar{\lambda} \Psi + \nabla s_2=0.
\end{equation}
Thus $s_1$ is constant on each connected component of $B'\setminus\overbar{D}$, and the first term on the right-hand side of \eqref{eq3.prf.lem.A} vanishes by the zero-flux condition $\langle\gamma_n u,1\rangle_{\Gamma_\ell}=0$ on every component $\Gamma_\ell$ of $\partial D$ due to $(u,B)\in \mathcal{X}^{\rm tgt}$. Moreover, by the Green formula and $\opdiv B=0$, 
\[
    \int_{\partial D}
    (B\cdot n)\,\overbar{s_2}\dd \sigma
    = 
    \int_{D}
    B\cdot\nabla\overbar{s_2} \dd x.
\]
Then, by the second equation of \eqref{eq0.prf.lem.A} and the rot Green formula, 
\[
\begin{split}
    \int_{D}
    B\cdot\nabla\overbar{s_2} \dd x
    &=
    -\frac{1}{\lambda}
    \int_{D}
    \oprot(\oprot B + H_* \times u)
    \cdot\nabla\overbar{s_2} \dd x \\
    &=
    \frac{1}{\lambda}
    \int_{\partial D}
    (
    (\oprot B + H_* \times u)
    \times n
    )
    \cdot\nabla\overbar{s_2}
    \dd \sigma \\
    &=
    -\int_{\partial D}
    ((\oprot B + H_*\times u)\times n)\cdot \overbar{\Psi}\dd \sigma.
\end{split}
\]
Here we note that the boundary pairing only involves the tangential trace of $\nabla\overbar{s_2}$; using $\nabla\overbar{s_2}=-\lambda\overbar{\Psi}$ on the exterior of $\partial D$ gives the last equality. Hence we have 
\[
    \int_{\partial D}
    (
    (\oprot B + H_*\times u)\times n
    )
    \cdot\overbar{\Psi}\dd \sigma
    +
    \int_{\partial D}
    (B\cdot n)\,\overbar{s_2}\dd \sigma
    =0.
\]
Hence $\langle U,V\rangle_D=0$ for any $V\in\mathcal{X}^\perp$, and therefore $U\in(\mathcal{X}^\perp)^\perp=\mathcal{X}$.

Finally, (\ref{item4.lem.A}) follows from the spectral theorem applied to the compact positive self-adjoint operator $A^* A$ on $\mathcal{K}^\perp$. This completes the proof of Lemma \ref{lem.A}.
\end{proof}

A more detailed computation in the proof above yields the following lemma.

\begin{lemma}\label{lem.bdry.int}
There exists $C>0$ depending on $\rho_0$ in \eqref{assump.lambda} such that the following holds: let $V=(v,H)\in \mathcal{X}^{\rm tgt}$ be associated with a pressure $q\in H^1(D)$. For $\mathsf{E}=(\mathsf{E}^{(1)},\mathsf{E}^{(2)})\in L^2(D)^6$, using $\Pi_D$ in \eqref{def.Pi}, we set
\[
    (\varphi,\Psi)
    = 
    (\overbar{\lambda} + \mathbb{S}_{\R^3}^*)^{-1}
    \Pi_D \mathsf{E}. 
\]
Then 
\begin{equation}
\begin{split}
	|\langle V,\mathsf{E}\rangle_D|
	&\le
	C \langle\lambda\rangle
	\Big(
	\|V\|_{H^2(D)} + \|q\|_{H^1(D)}
	\Big) \\
    &\quad\times
	\|(\varphi,\Psi)\|_{H^2(B')}^{3/4}
	\|(\varphi,\oprot\Psi)\|_{L^2(B'\setminus\overbar{D})}^{1/4}.
\end{split}
\end{equation}
\end{lemma}

\begin{proof}
Observe that there exist $s_1,s_2\in H^1(B')$ such that $(\varphi,\Psi)$ satisfies
\begin{equation}\label{eq1.prf.lem.bdry.int}
\begin{split}
    \left\{
    \begin{array}{ll}
    \overbar{\lambda} \varphi
    - \Delta\varphi
    + (\oprot\Psi)\times\tilde{H}_*
    + \nabla s_1
    = \mathtt{e}_D \mathsf{E}^{(1)} &\mbox{in}\ B', \\
    \overbar{\lambda} \Psi
    + \oprot (\oprot \Psi)
    + 
    \tilde{H}_*\cdot\nabla\varphi
    - (\nabla\tilde{H}_*)^\top\varphi
    + \nabla s_2
    = \mathtt{e}_D \mathsf{E}^{(2)} &\mbox{in}\ B', \\
    \opdiv \varphi
    = \opdiv \Psi
    = 0&\mbox{in}\ B'.
    \end{array}\right. 
\end{split}
\end{equation}
Applying the Green formula, exactly as in the proof of Lemma \ref{lem.A}, gives
\begin{equation}\label{eq2.prf.lem.bdry.int}
    \langle V,\mathsf{E}\rangle_D
    =
    \mathcal{B}_{{\rm Stokes}}(v,q,\varphi)
    + \mathcal{B}_{{\rm Mag}}(H,\Psi)
    + \mathcal{B}_{{\rm Coup}}(V,W)
    + \mathcal{B}_{{\rm pr},2}(H,s_2),
\end{equation}
where the terms on the right-hand side are defined in \eqref{def.B} and $W:=(\varphi,\Psi)$.

Below in the proof, we use the following trace interpolation estimate. All traces on $\partial D$ are taken with respect to the outward normal $n$ of $D$:
\begin{equation}\label{est1.prf.lem.bdry.int}
\begin{split}
    &\|\varphi\|_{H^{1/2}(\partial D)}
    + \|\partial_n\varphi\|_{H^{-1/2}(\partial D)}
    + \|(\oprot\Psi)\times n\|_{H^{-1/2}(\partial D)} \\
    &\le
    C\|(\varphi,\Psi)\|_{H^2(B')}^{3/4}\|(\varphi,\oprot\Psi)\|_{L^2(B'\setminus\overbar{D})}^{1/4}.
\end{split}
\end{equation}
This follows from the trace theorem on each connected component of $B'\setminus\overbar{D}$, interpolations, and an inequality $\|\oprot\Psi\|_{H^1(B'\setminus\overbar{D})}\le C\|\Psi\|_{H^2(B')}$.

We first estimate the Stokes term in \eqref{eq2.prf.lem.bdry.int}. From the equation, we see that 
\[
    \|\partial_n v-qn\|_{H^{-1/2}(\partial D)}
    +
    \|v\|_{H^{1/2}(\partial D)}
    \le
    C\langle\lambda\rangle
    \Big(
    \|V\|_{H^1(D)} + \|q\|_{H^1(D)}
    \Big).
\]
The term involving $s_1$ is treated as follows. Let $\mathcal{O}_0,\dots,\mathcal{O}_M$ be the connected components of $B'\setminus \overbar{D}$ on which $\mathtt{e}_D \mathsf{E}$ vanishes, and let $c_k^{(1)}$ be the mean value of $s_1$ over $\mathcal{O}_k$. Since $V\in\mathcal{X}^{\rm tgt}$ is assumed, the zero-flux condition \eqref{eq.flux} leads to
\[
    \int_{\Gamma_\ell}
    (v\cdot n)\,
    \overline{s_1}
    \dd \sigma
    =
    \int_{\Gamma_\ell}
    (v\cdot n)\,
    \Big(\overbar{s_1-c_\ell^{(1)}}\Big)
    \dd \sigma.
\]
In $\mathcal{O}_\ell$, the first equation of \eqref{eq1.prf.lem.bdry.int} reduces to
\[
    \nabla s_1
    =
    -\overbar{\lambda}\varphi
    + \Delta\varphi
    - (\oprot\Psi)\times \tilde{H}_*.
\]
Thus, by the Poincar\'{e} inequality and interpolation,
\[
\begin{split}
    \|s_1-c_\ell^{(1)}\|_{L^2(\mathcal{O}_\ell)}
    &\le
    C\|\nabla s_1\|_{H^{-1}(\mathcal{O}_\ell)} \\
    &\le
    C\langle\lambda\rangle
    \|(\varphi,\Psi)\|_{H^2(B')}^{1/2}
    \|(\varphi,\oprot\Psi)\|_{L^2(B'\setminus\overbar{D})}^{1/2},
\end{split}
\]
while the equation also gives
\[
    \|s_1-c_\ell^{(1)}\|_{H^1(\mathcal{O}_\ell)}
    \le
    C\langle\lambda\rangle
    \|(\varphi,\Psi)\|_{H^2(B')}.
\]
Interpolating these two estimates and taking the trace, we obtain
\[
    \sum_\ell
    \|s_1-c_\ell^{(1)}\|_{H^{-1/2}(\Gamma_\ell)}
    \le
    C\langle\lambda\rangle
    \|(\varphi,\Psi)\|_{H^2(B')}^{3/4}\|(\varphi,\oprot\Psi)\|_{L^2(B'\setminus\overbar{D})}^{1/4}.
\]
Hence, using \eqref{est1.prf.lem.bdry.int}, we have
\begin{equation}\label{est.stokes.prf.lem.bdry.int}
\begin{split}
    |\mathcal{B}_{{\rm Stokes}}(v,q,\varphi)|
    &\le
    C\langle\lambda\rangle
    \Big(
    \|V\|_{H^1(D)}
    +
    \|q\|_{H^1(D)}
    \Big) \\
    &\quad
    \times
    \|(\varphi,\Psi)\|_{H^2(B')}^{3/4}\|(\varphi,\oprot\Psi)\|_{L^2(B'\setminus\overbar{D})}^{1/4}.
\end{split}
\end{equation}

Next we estimate the magnetic boundary terms. Put
\[
    Z
    =
    \oprot H 
    + H_*\times v.
\]
The same computation as in the proof of Lemma \ref{lem.A} gives
\[
    (H_*\times v)\times n
    =
    -(v\cdot n)H_*
\]
on $\partial D$. Hence
\begin{equation}\label{eq3.prf.lem.bdry.int}
\begin{split}
    &\mathcal{B}_{{\rm Mag}}(H,\Psi)
    + \mathcal{B}_{{\rm Coup}}(V,W)
    + \mathcal{B}_{{\rm pr},2}(H,s_2) \\
    &=
    \int_{\partial D}
    (H\times n)\cdot \oprot\overbar{\Psi}\dd \sigma
    +
    \int_{\partial D}
    (Z\times n)\cdot \overbar{\Psi}\dd \sigma
    +
    \int_{\partial D}
    (H\cdot n)\overbar{s_2}\dd \sigma.
\end{split}
\end{equation}
The first term on the right-hand side of \eqref{eq3.prf.lem.bdry.int} is estimated by \eqref{est1.prf.lem.bdry.int}:
\begin{equation}\label{est2.prf.lem.bdry.int}
    \bigg|
    \int_{\partial D}
    (H\times n)\cdot\oprot\overbar{\Psi}\dd \sigma
    \bigg|
    \le
    C
    \|V\|_{H^2(D)}
    \|(\varphi,\Psi)\|_{H^2(B')}^{3/4}
    \|(\varphi,\oprot\Psi)\|_{L^2(B'\setminus\overbar{D})}^{1/4}.
\end{equation}
The last two terms in \eqref{eq3.prf.lem.bdry.int} are estimated as follows. Let $\mathcal{O}_0,\dots,\mathcal{O}_M$ be the connected components of $B'\setminus\overbar{D}$ again. For each $\ell$, the Helmholtz decomposition in $\mathcal{O}_\ell$ gives 
\[
    \Psi
    =
    \Psi_\ell
    + \nabla\chi_\ell,
    \qquad
    \Psi_\ell
    :=
    \mathbb{P}_{\mathcal{O}_\ell}(\Psi|_{\mathcal{O}_\ell}),
\]
with unique $\chi_\ell\in H^1(\mathcal O_\ell)$ up to additive constants. Let $\widetilde{\chi}_\ell\in H^1(D)$ be an extension of the trace of $\chi_\ell$ such that $\widetilde{\chi}_\ell=\chi_\ell$ on $\Gamma_\ell$ and $\widetilde{\chi}_\ell=0$ on the other connected components of $\partial D$. Since $Z\times n$ is tangential on $\partial D$, the boundary pairing below only depends on the tangential derivative of the trace of $\chi_\ell$. Hence, by the rot Green formula,
\[
\begin{split}
    \int_{\Gamma_\ell}
    (Z\times n)
    \cdot
    \nabla\overbar{\chi_\ell}
    \dd \sigma
    &=
    \int_{\partial D}
    (Z\times n)
    \cdot
    \nabla\overbar{\widetilde{\chi}_\ell}
    \dd \sigma \\
    &=
    -\int_D
    (\oprot Z)
    \cdot
    \nabla\overbar{\widetilde{\chi}_\ell}\dd x \\
    &=
    -\int_{\Gamma_\ell}
    (\oprot Z\cdot n)
    \overbar{\chi_\ell}\dd \sigma \\
    &=
    \lambda
    \int_{\Gamma_\ell}
    (H\cdot n)\overbar{\chi_\ell}\dd \sigma, 
\end{split}
\]
where $\oprot Z=-\lambda H$ in $D$ is used in the last line. Consequently,
\begin{equation}\label{eq4.prf.lem.bdry.int}
\begin{split}
    &\int_{\Gamma_\ell}
    (Z\times n)
    \cdot
    \overbar{\Psi}\dd \sigma
    +
    \int_{\Gamma_\ell}
    (H\cdot n)
    \overbar{s_2}
    \dd \sigma  \\
    &=
    \int_{\Gamma_\ell}
    (Z\times n)
    \cdot
    \overbar{\Psi_\ell}\dd \sigma
    +
    \int_{\Gamma_\ell}
    (H\cdot n)
    (\overbar{s_2} + \lambda\overbar{\chi_\ell})
    \dd \sigma.
\end{split}
\end{equation}
In $\mathcal O_\ell$, the second equation of \eqref{eq1.prf.lem.bdry.int} reduces to
\begin{equation}\label{eq5.prf.lem.bdry.int}
    \overbar{\lambda}\Psi_\ell
    + \nabla r_\ell
    =
    F_\ell,
\end{equation}
where
\[
    r_\ell
    :=
    s_2 + \overbar{\lambda}\chi_\ell,
    \qquad
    F_\ell
    :=
    -\big(
    \oprot\oprot\Psi
    +\tilde{H}_*\cdot\nabla\varphi
    -(\nabla\tilde{H}_*)^\top\varphi
    \big).
\]
Since $\tilde{H}_*\in W^{1,\infty}$ and $\opdiv\tilde{H}_*=0$, 
\[
    \|F_\ell\|_{H^{-1}(\mathcal O_\ell)}
    \le
    C\|(\varphi,\oprot\Psi)\|_{L^2(\mathcal O_\ell)}.
\]
Applying $\mathbb{P}_{\mathcal{O}_\ell}$ to \eqref{eq5.prf.lem.bdry.int}, we have
\[
    |\lambda|
    \|\Psi_\ell\|_{H^{-1}(\mathcal O_\ell)}
    \le
    C
    \|(\varphi,\oprot\Psi)\|_{L^2(\mathcal O_\ell)}.
\]
Thus, for $|\lambda|\ge\rho_0$, the Poincar\'{e} inequality and \eqref{eq5.prf.lem.bdry.int} give
\begin{equation}\label{est3.prf.lem.bdry.int}
\begin{split}
    \|\Psi_\ell\|_{H^{-1}(\mathcal O_\ell)}
    + \|r_\ell-c_\ell^{(2)}\|_{L^2(\mathcal O_\ell)}
    \le
    C\|(\varphi,\oprot\Psi)\|_{L^2(\mathcal O_\ell)},
\end{split}
\end{equation}
where $c_\ell^{(2)}$ denotes the mean value of $r_\ell$ over $\mathcal O_\ell$. Moreover, by elliptic regularity, 
\begin{equation}\label{est4.prf.lem.bdry.int}
    \|\Psi_\ell\|_{H^2(\mathcal O_\ell)}
    + \|r_\ell-c_\ell^{(2)}\|_{H^1(\mathcal O_\ell)}
    \le
    C\langle\lambda\rangle
    \|(\varphi,\Psi)\|_{H^2(B')}.
\end{equation}
Interpolating \eqref{est3.prf.lem.bdry.int} and \eqref{est4.prf.lem.bdry.int}, then taking traces, gives
\begin{equation}\label{est5.prf.lem.bdry.int}
\begin{split}
    &\|\Psi_\ell\|_{H^{1/2}(\Gamma_\ell)}
    +
    \|r_\ell-c_\ell^{(2)}\|_{H^{-1/2}(\Gamma_\ell)} \\
    &\le
    C\langle\lambda\rangle
    \|(\varphi,\Psi)\|_{H^2(B')}^{3/4}
    \|(\varphi,\oprot\Psi)\|_{L^2(B'\setminus\overbar{D})}^{1/4}.
\end{split}
\end{equation}
On the other hand, by the zero-flux condition \eqref{eq.flux},
\[
    \int_{\Gamma_\ell}
    (H\cdot n)\overbar{r_\ell}\dd \sigma
    =
    \int_{\Gamma_\ell}
    (H\cdot n)
    \Big(\overbar{r_\ell-c_\ell^{(2)}}\Big)
    \dd \sigma.
\]
Hence \eqref{eq4.prf.lem.bdry.int} and \eqref{est5.prf.lem.bdry.int} imply
\begin{equation}\label{est6.prf.lem.bdry.int}
\begin{split}
    &\bigg|
    \int_{\partial D}
    (Z\times n)\cdot\overbar{\Psi}\dd \sigma
    +
    \int_{\partial D}
    (H\cdot n)\overbar{s_2}\dd \sigma
    \bigg| \\
    &\le
    C \langle\lambda\rangle
    \|V\|_{H^2(D)}
    \|(\varphi,\Psi)\|_{H^2(B')}^{3/4}
    \|(\varphi,\oprot\Psi)\|_{L^2(B'\setminus\overbar{D})}^{1/4}.
\end{split}
\end{equation}
Combining \eqref{est2.prf.lem.bdry.int} and \eqref{est6.prf.lem.bdry.int}, we conclude that
\begin{equation}\label{est.mag.prf.lem.bdry.int}
\begin{split}
    &|
    \mathcal{B}_{{\rm Mag}}(H,\Psi)
    +
    \mathcal{B}_{{\rm Coup}}(V,W)
    +
    \mathcal{B}_{{\rm pr},2}(H,s_2)
    | \\
    &\le
    C\langle\lambda\rangle
    \|V\|_{H^2(D)}
    \|(\varphi,\Psi)\|_{H^2(B')}^{3/4}\|(\varphi,\oprot\Psi)\|_{L^2(B'\setminus\overbar{D})}^{1/4}.
\end{split}
\end{equation}

Finally, inserting \eqref{est.stokes.prf.lem.bdry.int} and \eqref{est.mag.prf.lem.bdry.int} into the equality \eqref{eq2.prf.lem.bdry.int} proves the lemma.
\end{proof}

\begin{remark}\label{rem.lem.bdry.int}
Note that \eqref{est1.prf.lem.bdry.int} and \eqref{est.lem.vor.interp} are stated for fixed geometries and in terms of inhomogeneous Sobolev norms. Under a simultaneous dilation of the relevant domains, both estimates admit the usual scale-covariant formulations with radius-weighted norms; the interpolation exponents $3/4$ and $1/4$ in \eqref{est1.prf.lem.bdry.int}, as well as the exponent $1/2$ in \eqref{est.lem.vor.interp}, remain unchanged. Since this is only a reformulation of the same estimates, it does not affect the $\eps$-dependence in Proposition \ref{prop.approx.resol} just below or the final rate in Theorem \ref{thm.main}.
\end{remark}

The following proposition is the main result of this section.

\begin{proposition}\label{prop.approx.resol}
There exists $C>0$ depending on $\rho_0$ in \eqref{assump.lambda} such that the following holds: let $V=(v,H)\in \mathcal{X}^{\rm tgt}$ be associated with a pressure $q\in H^1(D)$. Then, for $\eps > 0$, there exists $\mathsf{F}=(\mathsf{F}^{(1)}, \mathsf{F}^{(2)})\in L^2(Y)^3 \times L^2_\sigma(Y)$ such that, using $J_Y$ in \eqref{def.T.J}, 
\[
    V_1 
    := 
    \big(
    (\lambda + {\mathbb{S}}_{\R^3})^{-1} 
    J_Y \mathsf{F}
    \big)\big|_D
\]
approximates $V$ in $D$ as 
\[
    \|V - V_1\|_{L^2(D)}
    \le
    \eps 
    \Big( \|V\|_{H^{2}(D)} + \|q\|_{H^{1}(D)} \Big). 
\]
Moreover, 
\begin{equation}\label{est.F}
    \|\mathsf{F}\|_{L^2(Y)}
    \le
    \exp
    \bigg(
    C
    \Big(
    \frac{\exp(C \langle \lambda \rangle)}{\eps}
    \Big)^{4/\mu}
    \bigg)
    \|V\|_{L^2(D)},
\end{equation}
where $\mu$ is introduced in Lemma \ref{lem.stab}.
\end{proposition}

\begin{proof}
According to Lemma \ref{lem.A}, we expand $V\in \mathcal{X}$ as $V = \sum_{j=1}^{\infty} \beta_j V_j$ with $\beta_j = \langle V, V_j \rangle_{D}$. For a truncation parameter $N \in \N$, we define the approximate source $\mathsf{F}_N = \sum_{j=1}^{N} \alpha_j^{-1} \beta_j F_j$ and the approximation $V_N := A \mathsf{F}_N = \sum_{j=1}^{N} \beta_j V_j$. The error is given by $\mathsf{E}_N := V - V_N = \sum_{j=N+1}^{\infty} \beta_j V_j \in \mathcal{X}$. Our goal is to estimate $\|\mathsf{E}_N\|_{L^2(D)}$.

To improve the trivial bound $\|\mathsf{E}_N\|_{L^2(D)} \le \|V\|_{L^2(D)}$, we introduce the auxiliary function $(\varphi,\Psi)$. Denoting $\mathsf{E}_N=(\mathsf{E}_N^{(1)}, \mathsf{E}_N^{(2)})$, we define $(\varphi,\Psi)$ by, using $\Pi_D$ in \eqref{def.Pi}, 
\[
    (\varphi,\Psi)
    = 
    (\overbar{\lambda} + \mathbb{S}_{\R^3}^*)^{-1}
    \Pi_D \mathsf{E}_N. 
\]
Then, in terms of the adjoint $A^*$ in Lemma \ref{lem.A} (\ref{item1.lem.A}), we have
\[
    \big(
    \varphi|_{Y},
    \mathbb{P}_{Y}(\Psi|_{Y})
    \big) 
    = 
    A^* \mathsf{E}_N.
\]
Thus, we have the crucial smallness estimate in $Y$ as 
\begin{equation}\label{est1.prf.prop.approx.resol}
\begin{split}
    \|\varphi\|_{L^2(Y)}
    + \|\mathbb{P}_{Y}(\Psi|_{Y})\|_{L^2(Y)} 
    &\le
    C\|A^* \mathsf{E}_N\|_{L^2(Y)} \\
    &\le
    C\bigg\| \sum_{j=N+1}^{\infty} \alpha_j \beta_j F_j \bigg\|_{L^2(Y)}
    \le
    C\alpha_{N+1} \|\mathsf{E}_N\|_{L^2(D)}.
\end{split}
\end{equation}

Now we relate $\|\mathsf{E}_N\|_{L^2(D)}^2$ to boundary integrals. A key observation, relying on the orthogonality of the mode expansion, is that we can replace one $\mathsf{E}_N$ with $V$ as
\[
    \|\mathsf{E}_N\|_{L^2(D)}^2
    = \langle \mathsf{E}_N, \mathsf{E}_N \rangle_{D}
    = \langle V, \mathsf{E}_N \rangle_{D}.
\]
Then Lemma \ref{lem.bdry.int} yields that
\begin{equation}\label{est2.prf.prop.approx.resol}
\begin{split}
	\|\mathsf{E}_N\|_{L^2(D)}^2
    = 
    \langle V,\mathsf{E}_N \rangle_D
	&\le
	C\langle\lambda\rangle
	\Big(
	 \|V\|_{H^2(D)}
	+ \|q\|_{H^1(D)}
	\Big) \\
    &\quad\times
	\|(\varphi,\Psi)\|_{H^2(B')}^{3/4}
	\|(\varphi,\oprot\Psi)\|_{L^2(B'\setminus\overbar{D})}^{1/4}.
\end{split}
\end{equation}
Notice that, by the definition of $(\varphi,\Psi)$, the resolvent estimate gives
\begin{equation}\label{est3.prf.prop.approx.resol}
    \|(\varphi,\Psi)\|_{H^2(B')}
    \le
    C\|\mathsf{E}_N\|_{L^2(D)}.
\end{equation}

Let us propagate the smallness \eqref{est1.prf.prop.approx.resol} on $Y$ to \eqref{est2.prf.prop.approx.resol} on $B'\setminus \overbar{D}$. Let $\mathcal{O}_0,\dots,\mathcal{O}_M$ be the connected components of $B'\setminus \overbar{D}$. By the geometric assumption stated below \eqref{def.Y}, for each $k=0,\dots,M$ we can choose a nonempty subdomain $Y_k$ with $\overbar{Y_k}\subset Y\cap \mathcal{O}_k$. Since $\mathtt{e}_D \mathsf{E}_N=0$ in $\mathcal{O}_k$, we see that $(\varphi,\omega):=(\varphi,\oprot\Psi)$ solves the system 
\[
\begin{split}
    \left\{
    \begin{array}{ll}
    \overbar{\lambda} \varphi
    - \Delta\varphi
    + \omega\times\tilde{H}_*
    + \nabla s_1
    = 0&\mbox{in}\ \mathcal{O}_k, \\
    \overbar{\lambda} \omega
    - \Delta\omega
    + \oprot
    (
    \tilde{H}_*\cdot\nabla\varphi
    - (\nabla\tilde{H}_*)^\top\varphi
    )
    = 0&\mbox{in}\ \mathcal{O}_k, \\
    \opdiv \varphi
    = \opdiv \omega
    = 0&\mbox{in}\ \mathcal{O}_k.
    \end{array}\right. 
\end{split}
\]
Together with the estimate following from Lemma \ref{lem.vor.interp} and \eqref{est3.prf.prop.approx.resol}: 
\[
\begin{split}
    \|(\varphi,\omega)\|_{L^2(Y_k)}
    &\le
    \|\varphi\|_{L^2(Y)}
    + 
    \|\mathbb{P}_{Y}(\Psi|_{Y})\|_{L^2(Y)}^{1/2}
    \|\Psi\|_{H^2(Y)}^{1/2} \\
    &\le
    C(\alpha_{N+1} + \alpha_{N+1}^{1/2})
    \|\mathsf{E}_N\|_{L^2(D)}. 
\end{split}
\]
Lemma \ref{lem.stab} applied with $G=\mathcal{O}_k$ and $G'=Y_k$ yields that 
\begin{equation}\label{est4.prf.prop.approx.resol}
\begin{split}
	\|(\varphi,\omega)\|_{L^2(\mathcal{O}_k)}
    \le
    C 
    \exp(C \langle \lambda \rangle)\,
    \|\mathsf{E}_N\|_{L^2(D)}
    \bigg(\log \frac{1}{\alpha_{N+1}} \bigg)^{-\mu}. 
\end{split}
\end{equation}
Summing over $k=0,\ldots,M$, we conclude from \eqref{est2.prf.prop.approx.resol}--\eqref{est4.prf.prop.approx.resol} that 
\[
	\|\mathsf{E}_N\|_{L^2(D)}
    \le
	C
    \exp(C \langle \lambda \rangle)\,
	\Big(
	 \|V\|_{H^2(D)}
	+ \|q\|_{H^1(D)}
	\Big)
    \bigg(\log \frac{1}{\alpha_{N+1}} \bigg)^{-\mu/4}.
\]
Increasing the constant if necessary, we choose $N$ large enough to ensure 
\[
\begin{split}
    \frac{1}{\alpha_{N}}
    \le
    \exp
    \bigg(
    C^{4/\mu} 
    \Big(\frac{   \exp(C \langle \lambda \rangle)}{\eps} \Big)^{4/\mu}
    \bigg)
    <
    \frac{1}{\alpha_{N+1}}.
\end{split}
\]
Then we have 
\[
    \|\mathsf{E}_N\|_{L^2(D)} 
    \le 
    \eps 
    \Big( \|V\|_{H^{2}(D)} + \|q\|_{H^{1}(D)} \Big). 
\]
Consequently, the statement of Proposition \ref{prop.approx.resol} follows from setting $\mathsf{F} = \mathsf{F}_N$ and $V_1 = A \mathsf{F}$. The estimate \eqref{est.F} for $\mathsf{F}$ follows from $\|\mathsf{F}\|_{L^2(Y)}\le C(\alpha_{N})^{-1}\|V\|_{L^{2}(D)}$. 
\end{proof}

    \section{Global approximation of harmonic vector fields}
    \label{sec.harm}

In this section, we present the quantitative Runge approximation for the part $V_{{\rm topo}}$ in \eqref{eq.V.thm.main}.

Throughout this section, let
\[
    H
    \in 
    X_{{\rm har}}(D).
\]
We then observe that the pair $V=(0,H)$ and $q=H_*\cdot H$ solves
\begin{equation}\label{eq.MHD.topo.loc}
    \left\{
    \begin{array}{ll}
    -\Delta v 
    - H_*\cdot\nabla H - H\cdot\nabla H_* 
    + \nabla q
    = 0&\mbox{in}\ D, \\
    \oprot (\oprot H) + \oprot (H_* \times v)
    = 0&\mbox{in}\ D, \\
    \opdiv v
    = \opdiv H 
    = 0&\mbox{in}\ D.
    \end{array}\right. 
\end{equation}
Indeed, by the equality for vector fields $A,B$
\[
    \nabla(A \cdot B)
    =
    A\cdot \nabla B 
    + B\cdot \nabla A 
    + A\times (\oprot B) 
    + B\times (\oprot A), 
\]
and by $H_*,H\in X_{{\rm har}}(D)$, we see that the pair $V$ and $q$ satisfies \eqref{eq.MHD.topo.loc}.

Recall from \eqref{def.Y} that $Y$ is the control set. We start with the following problem:
\begin{equation}\label{eq.stat.topo.pde}
    \left\{
    \begin{array}{ll}
    -\Delta u
    - \tilde{H}_*\cdot\nabla B - B\cdot\nabla \tilde{H}_*
    + \nabla p
    = \mathtt{e}_Y i& \mbox{in}\ \R^3,\\
    \oprot(\oprot B) + \oprot(\tilde{H}_*\times u)
    = \mathtt{e}_Y j& \mbox{in}\ \R^3,\\
    \opdiv u
    = \opdiv B
    = 0&\mbox{in}\ \R^3.
    \end{array}
    \right.
\end{equation}

\begin{lemma}\label{lem.solver.topo}
If $\|\tilde{H}_*\|_{W^{1,\infty}(\R^3)}$ is sufficiently small, for every $F=(i,j)\in L^2(Y)^3 \times L^2_\sigma(Y)$, there exists a unique weak solution $U=U_F=(u_F,B_F)\in \dot{H}^1_\sigma(\R^3)^6$ of \eqref{eq.stat.topo.pde} with an associated pressure $p\in L^2_{loc}(\R^3)$ satisfying  
\[
    \|\nabla U\|_{L^2(\R^3)}
    \le
    C\|F\|_{L^2(Y)}. 
\]
\end{lemma}

\begin{remark}\label{rem.lem.solver.topo}
The same argument below applies, without any change, when the right-hand side $(\mathtt e_Y i,\mathtt e_Y j)$ is replaced by an arbitrary pair in $L^2(\R^3)^3\times L^2_\sigma(\R^3)$ with compact support. In addition, the argument applies to the adjoint form as well.
\end{remark}

\begin{proofx}{Lemma \ref{lem.solver.topo}}
For $U=(u,B),W=(\varphi,\Psi)$ in $\dot{H}^1_\sigma(\R^3)^6$, we define 
\[
\begin{split}
    a(U,W)
    &=
    \int_{\R^3} \nabla u \cdot \nabla\overbar{\varphi} \dd x
    + \int_{\R^3} \oprot B \cdot \oprot\overbar{\Psi} \dd x \\
    &\quad
    - \int_{\R^3}
    (
    \tilde{H}_*\cdot\nabla B
    + B\cdot\nabla \tilde{H}_*
    )
    \cdot \overbar{\varphi} \dd x
    + \int_{\R^3}
    (\tilde{H}_*\times u)\cdot \oprot\overbar{\Psi} \dd x.
\end{split}
\]
We also define the linear form
\[
    \ell_F(W)
    =
    \int_Y i\cdot \overbar{\varphi} \dd x
    + \int_Y j\cdot \overbar{\Psi} \dd x.
\]
Then the weak formulation of \eqref{eq.stat.topo.pde} is simplified to
\[
    a(U,W)=\ell_F(W).
\]
Since $Y$ is a bounded domain and $\dot H^1(\R^3)\hookrightarrow L^6(\R^3)$,
\[
    |\ell_F(W)|
    \le
    C\|F\|_{L^2(Y)}
    \|\nabla W\|_{L^2(\R^3)}.
\]
Moreover,
\[
    |a(U,W)|
    \le
    C\|\nabla U\|_{L^2(\R^3)}
    \|\nabla W\|_{L^2(\R^3)}.
\]
Since $B\in \dot H^1_\sigma(\R^3)$, the whole-space identity
\[
    \|\nabla B\|_{L^2(\R^3)}^2
    =
    \|\oprot B\|_{L^2(\R^3)}^2
\]
holds. Hence, for some $C_1,C_2>0$,
\[
\begin{split}
    \Re a(U,U)
    &\ge
    C_1
    \|\nabla U\|_{L^2(\R^3)}^2
    - C_2
    \|\tilde{H}_*\|_{W^{1,\infty}(\R^3)}
    \|\nabla U\|_{L^2(\R^3)}^2.
\end{split}
\]
Thus, if $\|\tilde{H}_*\|_{W^{1,\infty}(\R^3)}$ is sufficiently small, then
\[
    \Re a(U,U)
    \ge
    \frac12
    \|\nabla U\|_{L^2(\R^3)}^2.
\]
Hence the Lax-Milgram theorem yields a unique $U=(u,B)$ in $\dot H^1_\sigma(\R^3)^6$ satisfying
\[
    \|\nabla U\|_{L^2(\R^3)}
    \le
    C\|F\|_{L^2(Y)}.
\]
Recovering the pressure $p$ is routine work. The proof is complete. 
\end{proofx}

The following lemma is the stationary counterpart of Lemma \ref{lem.stab}.

\begin{lemma}\label{lem.stab.topo}
Let $G\subset\R^3$ be a bounded domain with smooth boundary, $G'$ be a subdomain of $G$ satisfying $\overbar{G'} \subset G$. Let $(\varphi,\omega)\in H^2_{{\rm loc}}(G)^6 \cap H^1(G)^6$ satisfy, for some $s\in H^1(G)$, 
\begin{equation}\label{eq.lem.stab.topo}
    \left\{
    \begin{array}{ll}
    -\Delta \varphi
    + \omega \times \tilde{H}_*
    + \nabla s
    = 0
    & \mbox{in}\ G,\\
    -\Delta \omega
    + \oprot
    (
    \tilde{H}_* \cdot\nabla \varphi 
    - (\nabla \tilde{H}_*)^\top \varphi
    )
    = 0
    & \mbox{in}\ G,\\
    \opdiv \varphi
    =
    \opdiv \omega
    = 0
    & \mbox{in}\ G.
    \end{array}
    \right.
\end{equation}
Suppose that, for $0<\eta<\mathcal{E}$, we have $\|(\varphi,\omega)\|_{H^1(G)} \le \mathcal{E}$ and $\|(\varphi,\omega)\|_{L^2(G')} \le \eta$. Then there exist $C,\mu>0$ independent of $\eta, \mathcal{E}$ such that 
\begin{equation}\label{est.lem.stab.topo}
    \|(\varphi,\omega)\|_{L^2(G)}
    \le 
    C \mathcal{E} 
    \bigg(\log \frac{\mathcal{E}}{\eta} \bigg)^{-\mu}. 
\end{equation}
In particular, if $(\varphi,\omega)=0$ in a nonempty open subset of $G$, then $(\varphi,\omega)=0$ in $G$.
\end{lemma}

\begin{proof}
The argument is exactly the version of the proof of Lemma \ref{lem.stab} with $\lambda=0$. 
\end{proof}

We now proceed with the approximation of $V=(0,H)$ by a global solution under forcing supported in $Y$ in \eqref{def.Y}. We define $\mathcal{X}^{\rm tgt}_0$ as the space of target functions by
\begin{equation}\label{def.Xtarget0}
    \mathcal{X}^{\rm tgt}_0
    = 
    \{V=(0,H)
    ~|~ 
    H\in X_{{\rm har}}(D)
    \}.
\end{equation}
We also define the linear operator 
\[
    T_0: L^2(Y)^3 \times L^2_\sigma(Y) \to L^2(D)^6
\]
as follows: for $F\in L^2(Y)^3 \times L^2_\sigma(Y)$, letting $U$ be the weak solution in Lemma \ref{lem.solver.topo}, 
\[
    T_0 F 
    =
    U|_D.
\]
Moreover, 
\begin{equation}\label{def.X0}
    \mathcal{X}_{0}
    =
    \overbar{\opRan(T_0)}^{L^2(D)^6}.
\end{equation}
Let us denote by $\mathcal{K}_{0}$ the kernel of $T_0$ and by $\mathcal{K}_{0}^\perp$ the orthogonal space of $\mathcal{K}_{0}$ in $L^2(Y)^3 \times L^2_\sigma(Y)$. Finally, let us denote by $A_{0}$ the restriction of $T_0$ to $\mathcal{K}_{0}^\perp$.

\begin{lemma}\label{lem.Atopo}
The following hold. 
\begin{enumerate}[(1)]
\item\label{item1.lem.Atopo}
For $V=(v,H)\in \mathcal{X}_{0}$, let $(\varphi,\Psi)$ be the unique weak solution of the adjoint system
\begin{equation}\label{eq.stat.adjoint}
    \left\{
    \begin{array}{ll}
    - \Delta \varphi
    + (\oprot \Psi) \times \tilde{H}_*
    + \nabla s_1
    = \mathtt{e}_D v
    & \mbox{in}\ \R^3,\\
    \oprot(\oprot \Psi)
    + \tilde{H}_* \cdot\nabla \varphi 
    - (\nabla \tilde{H}_*)^\top \varphi
    + \nabla s_2
    = \mathtt{e}_D H
    & \mbox{in}\ \R^3,\\
    \opdiv \varphi
    =
    \opdiv \Psi
    = 0
    & \mbox{in}\ \R^3.
    \end{array}
    \right.
\end{equation}
Then the adjoint $A_{0}^*: \mathcal{X}_{0} \to \mathcal{K}_{0}^\perp$ of $A_{0}$ is given by
\[
    A_{0}^* V
    =
    \big(
    \varphi|_{Y},
    \mathbb{P}_{Y}(\Psi|_{Y})
    \big).
\]

\item\label{item2.lem.Atopo}
The mapping $A_{0}^* A_{0}$ is positive, compact, and self-adjoint on $\mathcal{K}_{0}^\perp$.

\item\label{item3.lem.Atopo}
The range of $A_{0}$ is dense in $\mathcal{X}_{0}$. Moreover, $\mathcal{X}^{\rm tgt}_0$ is contained in $\mathcal{X}_0$.

\item\label{item4.lem.Atopo}
There exist orthonormal bases 
\[
    \{F_j\}_{j=1}^{\infty} \subset \mathcal{K}_{0}^\perp, 
    \qquad
    \{V_j\}_{j=1}^{\infty} \subset \mathcal{X}_{0}, 
\]
and positive constants $\{\alpha_j\}_{j=1}^{\infty}$ with $\alpha_1\ge \alpha_2\ge\cdots>0$ and $\alpha_j\to0$ such that 
\[
    A_{0} F_j = \alpha_j V_j, 
    \qquad
    A_{0}^* V_j = \alpha_j F_j. 
\]
\end{enumerate}
\end{lemma}

\begin{proof}
The proof of (\ref{item1.lem.Atopo}) is quite similar to that of Lemma \ref{lem.A} (\ref{item1.lem.A}) and thus omitted.

For (\ref{item2.lem.Atopo}), elliptic regularity in a neighborhood of $\overbar{D}$, where the forcing vanishes, gives
\[
    \|T_0F\|_{H^2(D)}
    \le
    C\|F\|_{L^2(Y)}.
\]
Thus $A_0$ is compact into $L^2(D)^6$, and $A_0^*A_0$ is positive, compact, and self-adjoint.

We now prove (\ref{item3.lem.Atopo}). The proof is similar to that of Lemma \ref{lem.A} (\ref{item3.lem.A}). It suffices to show that  $\mathcal{X}^{\rm tgt}_0\subset\mathcal{X}_0$. For $V=(v,H)\in L^2(D)^6$ orthogonal to $\mathcal{X}_{0}$, let $(\varphi,\Psi)$ be the solution of \eqref{eq.stat.adjoint}. Then we see that, for every $F\in L^2(Y)^3\times L^2_\sigma(Y)$, 
\[
    0
    =
    \langle T_0 F,V\rangle_D 
    =
    \big\langle
    F,
    \big(
    \varphi|_{Y},
    \mathbb{P}_{Y}(\Psi|_{Y})
    \big)
    \big\rangle_{Y}. 
\]
Consequently,
\[
    \varphi|_Y=0,
    \qquad
    \mathbb{P}_Y(\Psi|_Y)=0.
\]
Set $\omega=\oprot\Psi$. The second equality implies that $\Psi|_Y$ is locally of gradient form, and hence $\omega|_Y=0$. Let $\mathcal{O}_0,\ldots,\mathcal{O}_M$ be the connected components of $B'\setminus\overbar D$, and choose nonempty subdomains $Y_k\Subset Y\cap\mathcal{O}_k$. Since the forcing in \eqref{eq.stat.adjoint} vanishes in $\mathcal{O}_k$, the pair $(\varphi,\omega)$ satisfies \eqref{eq.lem.stab.topo}. Lemma \ref{lem.stab.topo} therefore gives $(\varphi,\omega)=0$ in $\mathcal{O}_k$ for every $k$, and thus 
\[
    \varphi|_{B'\setminus\overbar{D}}=0,
    \qquad
    \omega|_{B'\setminus\overbar{D}}=0.
\]

Let $U=(0,B)\in \mathcal{X}^{\rm tgt}_0$ be associated with a pressure $p = H_*\cdot B$. Then the Green formula used in the proof of Lemma \ref{lem.A} (\ref{item3.lem.A}) yields
\[
    \langle U,V\rangle_D
    =
    -\int_{\partial D}
    p n\cdot\overbar{\varphi}\dd \sigma
    +
    \int_{\partial D}
    (B\times n)\cdot\overbar{\omega}\dd \sigma.
\]
Since $(\varphi,\Psi)\in H^2(B')^6$ by elliptic regularity, the vanishing of $(\varphi,\omega)$ in the exterior gives zero traces in this identity. Hence $\langle U,V\rangle_D=0$, and therefore $U\in(\mathcal{X}_0^\perp)^\perp=\mathcal{X}_0$.

Finally, (\ref{item4.lem.Atopo}) follows from the spectral theorem applied to the compact positive self-adjoint operator $A_0^* A_0$ on $\mathcal{K}_0^\perp$. This completes the proof of Lemma \ref{lem.Atopo}.
\end{proof}

The following lemma is the stationary counterpart of Lemma \ref{lem.bdry.int}, specialized to the harmonic targets needed in the proof of Theorem \ref{thm.main}.

\begin{lemma}\label{lem.bdry.int.topo}
There exists $C>0$ such that the following holds: let $V=(0,H)\in \mathcal{X}^{\rm tgt}_0$ be associated with a pressure $q = H_*\cdot H$. For $\mathsf{E}\in L^2(D)^6$, we let $(\varphi,\Psi)$ be the unique weak solution of \eqref{eq.stat.adjoint} with the forcing $\mathtt{e}_D \mathsf{E}$. Then
\begin{equation}\label{est.lem.bdry.int.topo}
\begin{split}
    |\langle V,\mathsf{E}\rangle_D|
    &\le
    C\|H\|_{L^2(D)}
    \|(\varphi,\Psi)\|_{H^2(B')}^{3/4}
    \|(\varphi,\oprot\Psi)\|_{L^2(B'\setminus\overbar D)}^{1/4}.
\end{split}
\end{equation}
\end{lemma}

\begin{proof}
Applying the same Green formula as in Lemma \ref{lem.bdry.int} gives
\[
    \langle V,\mathsf{E}\rangle_D
    =
    -\int_{\partial D}
    q n\cdot\overbar{\varphi}\dd \sigma
    +
    \int_{\partial D}
    (H\times n)\cdot\oprot\overbar{\Psi}\dd \sigma.
\]
Indeed, $v=0$, $\oprot H=0$, and $H\cdot n=0$ on $\partial D$ eliminate the other terms.

The trace interpolation estimate
\[
\begin{split}
    &\|\varphi\|_{H^{1/2}(\partial D)}
    + \|(\oprot\Psi)\times n\|_{H^{-1/2}(\partial D)} \\
    &\le
    C
    \|(\varphi,\Psi)\|_{H^2(B')}^{3/4}
    \|(\varphi,\oprot\Psi)\|_{L^2(B'\setminus\overbar D)}^{1/4}
\end{split}
\]
follows as in \eqref{est1.prf.lem.bdry.int}.  Since $X_{{\rm har}}(D)$ is finite-dimensional and $q = H_*\cdot H$, we have
\[
    \|H\|_{H^1(D)}+\|q\|_{H^1(D)}
    \le
    C\|H\|_{L^2(D)}.
\]
Combining these estimates yields the desired estimate \eqref{est.lem.bdry.int.topo}.
\end{proof}

The following proposition is the main result of this section.

\begin{proposition}\label{prop.approx.topo}
There exists $C>0$ such that the following holds: let $V=(0,H)\in \mathcal{X}^{\rm tgt}_0$. Then, for $\eps>0$, there exists $\mathsf F\in L^2(Y)^3\times L^2_\sigma(Y)$ such that $U=U_{\mathsf{F}}$ in Lemma \ref{lem.solver.topo} approximates $V$ in $D$ as 
\[
    \|V - U\|_{L^2(D)}
    \le
    \eps 
    \|H\|_{L^{2}(D)}. 
\]
Moreover, 
\begin{equation}\label{est.F.topo}
    \|\mathsf{F}\|_{L^2(Y)}
    \le
    \exp
    \bigg(
    C
    \Big(\frac{1}{\eps} \Big)^{4/\mu}
    \bigg)
    \|H\|_{L^2(D)}, 
\end{equation}
where $\mu$ is introduced in Lemma \ref{lem.stab.topo}. 
\end{proposition}

\begin{proof}
By Lemma \ref{lem.Atopo} (\ref{item3.lem.Atopo}), we have
$V\in\mathcal{X}_0$. We then apply the proof of Proposition
\ref{prop.approx.resol} to the operator $A_0$, using Lemma \ref{lem.Atopo} (\ref{item4.lem.Atopo}), with
Lemma \ref{lem.stab.topo} and Lemma \ref{lem.bdry.int.topo} in
place of Lemma \ref{lem.stab} and Lemma \ref{lem.bdry.int},
respectively. The details are omitted to avoid repetition. 
\end{proof}

    \section{Proof of Theorem \ref{thm.main}}
    \label{sec.prf}

In this section, we prove Theorem \ref{thm.main}. We first approximate the time-evolving part by the Dunford integral representation together with the resolvent approximation in Section \ref{sec.resol}, and then we treat the stationary topological part by the result of Section \ref{sec.harm}. We conclude by combining the two constructions.

Let $V(t)$ be a solution to the linearized MHD system \eqref{eq.MHD} written as 
\begin{equation}\label{eq.V}
    V(t) =  V_{{\rm evol}}(t) + V_{{\rm topo}},
    \qquad
    V_{{\rm evol}}(t) = e^{-t {\mathbb{S}}} V_0,
    \qquad
    V_{{\rm topo}} = (0, H_*'),
\end{equation}
for $V_0\in D({\mathbb{S}})$ and $H_*^\prime\in X_{{\rm har}}(D)$.

Recall that $\tilde{H}_*$ denotes the extended background field in Lemma \ref{lem.ext.H}.

    \paragraph{Approximation of the time-evolving part.}

We start by approximating the time-evolving part $V_{{\rm evol}}$ in \eqref{eq.V}. First we fix a solenoidal whole-space extension of the initial data. Since $V_0=(v_0,H_0)\in D(\mathbb{S})$, both $v_0$ and $H_0$ are divergence-free $H^2$-fields on $D$ with zero normal trace on $\partial D$. By the divergence-free extension theorem of Kato-Mitrea-Ponce-Taylor \cite[Corollary 3.2]{KatoMitreaPonceTaylor2000}, applied with $s=2$ and $p=2$, there exists $\tilde{V}_0=(\tilde{v}_0,\tilde{H}_0)\in H^2_\sigma(\R^3)^3\times H^2_\sigma(\R^3)^3$ with compact support in a fixed ball containing $\overbar{D}$ such that
\begin{equation}\label{est1.ext.V0}
    \tilde{V}_0|_D = V_0,
    \qquad
    \|\tilde{V}_0\|_{L^2(\R^3)}
    \le
    C\|V_0\|_{L^2(D)},
    \qquad
    \|\tilde{V}_0\|_{H^2(\R^3)}
    \le
    C\|V_0\|_{H^2(D)}.
\end{equation}
By elliptic regularity for the Stokes block and \cite[Lemma 2.2]{KozonoShimizuYanagisawa2025} for the magnetic block, the graph norm of $\mathbb{S}$ is equivalent on $D(\mathbb{S})$ to the $H^2(D)$-norm. Since moreover $0\in\rho(-\mathbb{S})$, 
\begin{equation}\label{est2.ext.V0}
    \|V_0\|_{H^2(D)}
    \le
    C\|\mathbb{S} V_0\|_{L^2(D)}.
\end{equation}

Instead of approximating $V_{{\rm evol}}$ directly, we introduce the reference solution
\[
    U_{{\rm ref}}(t) 
    := 
    e^{-t\mathbb{S}_{\R^3}}\tilde{V}_0, 
\]
where $\mathbb{S}_{\R^3}$ is the global MHD operator in Section \ref{sec.ext}, and consider the difference 
\[
    E(t) 
    := 
    V_{{\rm evol}}(t) - U_{{\rm ref}}(t).
\]
Fix $\rho_0\in(0,1)$ arbitrarily. Using the Dunford integral, we represent $E(t)$ as
\[
    E(t) 
    = 
    \frac{1}{2\pi \ii} 
    \int_{\gamma} 
    \exp(\lambda t) V(\lambda) \dd \lambda, 
    \qquad
    V(\lambda) 
    := (\lambda + \mathbb{S})^{-1} V_0 
    - \big((\lambda + \mathbb{S}_{\R^3})^{-1} \tilde{V}_0\big)|_D. 
\]
Here $\gamma$ is chosen as follows. By the sectoriality of both $\mathbb{S}$ and $\mathbb{S}_{\R^3}$, there exists $\delta\in(0,\pi/2)$ such that
$
    \overbar{\Sigma_{\pi-\delta}}\setminus\{0\}
    \subset
    \rho(-\mathbb{S})\cap\rho(-\mathbb{S}_{\R^3})
$
holds; see Remark \ref{rem.smallness} and \cite[Proposition 1.2.13]{Lunardi1995}. We then let $\gamma$ denote the oriented (counterclockwise) contour
\[
    \gamma
    = 
    \Big\{
    |\oparg z|=\pi-\delta, \mkern9mu |z|\ge \rho_0 
    \Big\} 
    \cup 
    \Big\{
    |\oparg z| \le \pi-\delta, \mkern9mu |z|=\rho_0 
    \Big\}. 
\]

We split the integral into high- and low-frequency parts with a parameter $L > \rho_0$:
\[
\begin{split}
    E(t) 
    &:=
    E_{{\rm high}}(t) + E_{{\rm low}}(t) \\
    &:=
    \frac{1}{2\pi \ii} 
    \int_{\gamma \cap \{|\lambda| > L\}} 
    \exp(\lambda t) V(\lambda) \dd \lambda
    + 
    \frac{1}{2\pi \ii} 
    \int_{\gamma \cap \{|\lambda| \le L\}} 
    \exp(\lambda t) V(\lambda) \dd \lambda. 
\end{split}
\]
For the high-frequency part $E_{{\rm high}}$, we apply the argument in \cite{Higaki2025_Runge}. We write 
\[
    V(\lambda) 
    = -\frac{1}{\lambda} 
    \Big( 
    (\lambda+\mathbb{S})^{-1} \mathbb{S}V_0 
    - \big((\lambda+\mathbb{S}_{\R^3})^{-1} \mathbb{S}_{\R^3}\tilde{V}_0\big)|_D 
    \Big).
\]
This, together with the standard sectorial resolvent estimate $\|(\lambda+\mathbb{S})^{-1}\|\le C|\lambda|^{-1}$ and its analogue for $\mathbb{S}_{\R^3}$ (see \cite[Proposition 1.2.13]{Lunardi1995} and \cite[Chapter 2]{Pazy1983}) yields 
\[
\begin{split}
    \|V(\lambda)\|_{L^2(D)} 
    \le
    \frac{C}{|\lambda|^2} 
    \big( 
    \|\mathbb{S} V_0\|_{L^2(D)} 
    + \|\mathbb{S}_{\R^3} \tilde{V}_0\|_{L^2(\R^3)} 
    \big) 
    \le
    \frac{C}{|\lambda|^2} 
    \|\mathbb{S} V_0\|_{L^2(D)}, 
\end{split}
\]
where we have used \eqref{est1.ext.V0}--\eqref{est2.ext.V0}. Consequently, $E_{{\rm high}}$ is estimated as
\begin{equation}\label{est.E.high}
\begin{split}
    \|E_{{\rm high}}(t)\|_{L^2(D)}
    &\le
    C\|\mathbb{S} V_0\|_{L^2(D)}
    \int_{\gamma\,\cap \{|\lambda|>L\}}
    \frac{\exp(t \Re \lambda)}{|\lambda|^2}
    \dd |\lambda| \\
    &\le
    \frac{C\|\mathbb{S} V_0\|_{L^2(D)}}{L}
    \exp
    (-tL\cos\delta),
    \quad
    t \ge 0.
\end{split}
\end{equation}

For the low-frequency part $E_{{\rm low}}$, we set
\[
    \gamma_L = \gamma \cap \{|\lambda|\le L\}.
\]
For every $\lambda\in \gamma_L$, using $J_Y$ in \eqref{def.T.J}, we define the operator
\[
    T_\lambda F
    =
    \big(
    (\lambda + {\mathbb{S}}_{\R^3})^{-1} 
    J_Y F
    \big)\big|_D,
    \quad
    F\in L^2(Y)^3\times L^2_\sigma(Y).
\]
The map $\lambda\mapsto T_\lambda$ is analytic on $\gamma_L$ as a bounded operator from $L^2(Y)^3\times L^2_\sigma(Y)$ to $L^2(D)^6$. Since $\tilde{V}_0|_D=V_0$ and $\tilde{H}_*=H_*$ on $D$, the field $V(\lambda)$ satisfies the homogeneous local resolvent system \eqref{eq.MHD.resol.loc} in $D$. Moreover, it holds that $V(\lambda)\in \mathcal{X}^{\rm tgt}$ by a similar argument when extending $V_0$ to $\tilde{V}_0$. In addition, the resolvent estimates and \eqref{est1.ext.V0}--\eqref{est2.ext.V0} imply
\begin{align*}
    \|V(\lambda)\|_{H^2(D)}
    +
    \|q(\lambda)\|_{H^1(D)}
    &\le
    C \langle \lambda\rangle
    \big(\|V_0\|_{L^2(D)}+\|\tilde{V}_0\|_{L^2(\R^3)}\big)\\
    &\le
    C \langle \lambda\rangle
    \|\mathbb{S} V_0\|_{L^2(D)},
    \quad
    \lambda\in\gamma_L,
\end{align*}
with the pressure normalized by $\int_D q(\lambda)\dd x=0$. Therefore Proposition \ref{prop.approx.resol} yields that, for every $\eta>0$ and every $\lambda\in \gamma_L$, there exists $F_\lambda^\sharp$ in $L^2(Y)^3\times L^2_\sigma(Y)$ such that
\[
    \|V(\lambda)-T_\lambda F_\lambda^\sharp\|_{L^2(D)}
    \le
    C \eta 
    \langle \lambda\rangle 
    \|\mathbb{S} V_0\|_{L^2(D)},
\]
satisfying, using resolvent estimates and \eqref{est1.ext.V0}, 
\[
    \|F_\lambda^\sharp\|_{L^2(Y)}
    \le
    M \|V_0\|_{L^2(D)},
    \qquad
    M
    :=
    \exp
    \bigg(
    C
    \Big(
    \frac{\exp(C\langle L \rangle)}{\eta}
    \Big)^{4/\mu}
    \bigg).
\]
We now replace the noncanonical choice $F_\lambda^\sharp$ by a canonical one that varies continuously with $\lambda$. For each $\lambda$ in $\gamma_L$ and each $\tau>0$, we consider the Tikhonov functional
\begin{equation*}
    J_{\lambda,\tau}(F)
    :=
    \|V(\lambda)-T_\lambda F\|_{L^2(D)}^2
    + \tau \|F\|_{L^2(Y)}^2,
    \quad
    F \in L^2(Y)^3\times L^2_\sigma(Y).
\end{equation*}
This functional is strictly convex and coercive, hence it has a unique minimizer. Since its Euler-Lagrange equation reads, letting $T_\lambda^*$ denote the adjoint of $T_\lambda$, 
\begin{equation*}
    (T_\lambda^* T_\lambda + \tau I) 
    \mathsf{F}_\tau(\lambda)
    =
    T_\lambda^* V(\lambda),
\end{equation*}
the minimizer is equivalently given by
\begin{equation*}
    \mathsf{F}_\tau(\lambda)
    =
    T_\lambda^*
    (T_\lambda T_\lambda^* + \tau I)^{-1}
    V(\lambda).
\end{equation*}
Choosing $\tau = \eta^2 L^2/M^2$, we see that
\begin{equation*}
\begin{split}
    J_{\lambda,\tau}(F_\lambda^\sharp)
    &\le
    \|V(\lambda)-T_\lambda F_\lambda^\sharp\|_{L^2(D)}^2
    + \tau \|F_\lambda^\sharp\|_{L^2(Y)}^2 \\
    &\le
    C \eta^2 L^2 \|\mathbb{S} V_0\|_{L^2(D)}^2
    + \tau M^2 \|V_0\|_{L^2(D)}^2 \\
    &\le
    C \eta^2 L^2 \|\mathbb{S} V_0\|_{L^2(D)}^2.
\end{split}
\end{equation*}
From now on we set
\begin{equation*}
    \mathsf{F}(\lambda)
    :=
    \mathsf{F}_\tau(\lambda),
    \qquad
    U^\prime_{{\rm corr}}(\lambda)
    :=
    (\lambda + \mathbb{S}_{\R^3})^{-1}
    J_Y \mathsf{F}(\lambda).
\end{equation*}
By minimality, 
$
    J_{\lambda,\tau}(\mathsf{F}(\lambda))
    \le
    J_{\lambda,\tau}(F_\lambda^\sharp)
$. Thus we have
\begin{equation*}
    \|V(\lambda)-U^\prime_{{\rm corr}}(\lambda)\|_{L^2(D)}
    \le
    C \eta 
    L 
    \|\mathbb{S} V_0\|_{L^2(D)}
\end{equation*}
and, using the definition of $\tau$,
\begin{equation*}
    \|\mathsf{F}(\lambda)\|_{L^2(Y)}
    \le
    C \tau^{-1/2} 
    \eta
    L \|\mathbb{S} V_0\|_{L^2(D)}
    \le
    C M
    \|\mathbb{S} V_0\|_{L^2(D)}.
\end{equation*}
The point of this replacement is that the minimizer is now canonical. Since $\lambda\mapsto T_\lambda$ and $\lambda\mapsto V(\lambda)$ are analytic on $\gamma_L$, the map $\lambda\mapsto T_\lambda T_\lambda^* + \tau I$ is continuous and takes values in boundedly invertible operators. By the explicit formula for the minimizer and continuity of inversion on the open set of boundedly invertible operators, it follows that
\begin{equation*}
    \lambda\mapsto \mathsf{F}(\lambda),
    \qquad
    \lambda\mapsto U^\prime_{{\rm corr}}(\lambda)
\end{equation*}
are continuous on $\gamma_L$. In particular, the Bochner integrals below are well defined.

Using $U^\prime_{{\rm corr}}(\lambda)$, we define the correction term in the time domain as
\[
    U_{{\rm corr}}(t)
    =
    \frac{1}{2\pi \ii}
    \int_{\gamma_L}
    \exp(\lambda t)
    U^\prime_{{\rm corr}}(\lambda) 
    \dd \lambda.
\]
Then the error is bounded by
\begin{equation}\label{est.E.low}
\begin{split}
    \|E_{{\rm low}}(t) - U_{{\rm corr}}(t)\|_{L^2(D)}
    &\le
    \int_{\gamma_L}
    \exp(t \Re \lambda)
    \|V(\lambda)-U^\prime_{{\rm corr}}(\lambda)\|_{L^2(D)}
    \dd |\lambda| \\
    &\le
    C \eta L \|\mathbb{S} V_0\|_{L^2(D)}
    \int_{\gamma_L}
    \exp(t \Re \lambda)
    \dd |\lambda| \\
    &\le
    C \eta L^2 \|\mathbb{S} V_0\|_{L^2(D)} \exp(\rho_0 t),
    \quad
    t \ge 0.
\end{split}
\end{equation}

Set
\[
    U_{{\rm evol}}(t) = U_{{\rm ref}}(t) + U_{{\rm corr}}(t).
\]
Then $U_{{\rm evol}}=(u,B)$ solves
\begin{equation}\label{eq.MHD.evol}
\left\{
    \begin{array}{ll}
    \partial_{t} u - \Delta u
    - \tilde{H}_*\cdot\nabla B - B\cdot\nabla \tilde{H}_*
    + \nabla p
    = \mathtt{e}_Y f&\mbox{in}\ \R^3\times (0,\infty), \\
    \partial_{t} B + \oprot (\oprot B) + \oprot (\tilde{H}_* \times u)
    = \mathtt{e}_Y g&\mbox{in}\ \R^3\times (0,\infty), \\
    \opdiv u
    = \opdiv B
    = 0&\mbox{in}\ \R^3\times [0,\infty).
    \end{array}\right.
\end{equation}
Here the forcing $(f,g)$ is defined by
\[
    (f,g)(t)
    =
    \frac{1}{2\pi \ii}
    \int_{\gamma_L}
    \exp(\lambda t) \mathsf{F}(\lambda) \dd \lambda.
\]
Since $\gamma_L$ is compact and 
$
\lambda\mapsto \mathsf F(\lambda)
$
is continuous, differentiation under the integral sign shows that $(f,g)$ belongs to $C^\infty([0,\infty);L^2(Y)^3\times L^2_\sigma(Y))$. It is estimated as
\begin{equation}\label{est.fg}
\begin{split}
    \|(f,g)(t)\|_{L^2(Y)}
    &\le
    \int_{\gamma_L}
    \exp(t \Re \lambda)
    \|\mathsf{F}(\lambda)\|_{L^2(Y)}
    \dd |\lambda| \\
    &\le
    C M \|\mathbb{S} V_0\|_{L^2(D)}
    \int_{\gamma_L}
    \exp(t \Re \lambda)
    \dd |\lambda| \\
    &\le
    \exp
    \bigg(
    C
    \Big(
    \frac{\exp(C\langle L \rangle)}{\eta}
    \Big)^{4/\mu}
    \bigg)
    \|\mathbb{S} V_0\|_{L^2(D)} \exp(\rho_0 t),
    \quad
    t \ge 0.
\end{split}
\end{equation}

    \paragraph{Approximation of the topological part.}

We next approximate the topological part $V_{{\rm topo}}$ in \eqref{eq.V}. We apply Proposition \ref{prop.approx.topo}: for any $\eps>0$, there exists $(i,j)\in L^2(Y)^3\times L^2_\sigma(Y)$ such that the solution $U_{{\rm topo}}=(u,B)$ of
\begin{equation}\label{eq.MHD.topo}
\left\{
    \begin{array}{ll}
    - \Delta u 
    - \tilde{H}_*\cdot\nabla B - B\cdot\nabla \tilde{H}_* 
    + \nabla p 
    = \mathtt{e}_Y i&\mbox{in}\ \R^3, \\
    \oprot (\oprot B) + \oprot (\tilde{H}_* \times u)
    = \mathtt{e}_Y j&\mbox{in}\ \R^3, \\
    \opdiv u 
    = \opdiv B 
    = 0&\mbox{in}\ \R^3, 
    \end{array}\right. 
\end{equation}
approximates $V_{{\rm topo}}$ in $D$ as
\begin{equation}\label{est.H.B}
    \|V_{{\rm topo}} - U_{{\rm topo}}\|_{L^2(D)}
    <\eps \|H_*'\|_{L^2(D)}.
\end{equation}
Moreover, 
\begin{equation}\label{est.ij}
    \|(i,j)\|_{L^2(Y)}
    \le
    \exp
    \bigg(
    C
    \Big(\frac{1}{\eps} \Big)^{4/\mu}
    \bigg)
    \|H_*'\|_{L^2(D)}. 
\end{equation}

    \paragraph{Conclusion.}

Now we define the total global approximation $U=U(t)$ by
\[
    U(t) = U_{{\rm evol}}(t) + U_{{\rm topo}}.
\]
Then $U$ solves \eqref{eq.thm.main} due to \eqref{eq.MHD.evol} and \eqref{eq.MHD.topo}. Combining the high and low frequency estimates \eqref{est.E.high} and \eqref{est.E.low}, respectively, with \eqref{est.H.B}, and then using \eqref{est1.ext.V0}--\eqref{est2.ext.V0}, we obtain
\[
\begin{split}
    &\|V(t) - U(t)\|_{L^2(D)} \\
    &\le
    C\bigg(
    \frac{1}{L}
    \exp(-tL\cos\delta)
    + \eta L^2 \exp(\rho_0 t)
    \bigg)
    \|\mathbb{S} V_0\|_{L^2(D)}
    + C\eps \|H_*^\prime\|_{L^2(D)}. 
\end{split}
\]

Choosing $L=\eps^{-1}$ and $\eta=\eps^3$, we arrive at \eqref{est.thm.main}. The estimate of the forcing $(f,g)$ is verified by \eqref{est.fg} and that of $(i,j)$ by \eqref{est.ij}. The proof of Theorem \ref{thm.main} is complete.
\hfill $\Box$

    \appendix

    \section{Proof of \eqref{eq2.prf.lem.A}}
    \label{appx.bdry.int}

In this appendix, we derive the equality \eqref{eq2.prf.lem.A}.

\begin{proofx}{\eqref{eq2.prf.lem.A}}
Since $\tilde{H}_*=H_*$ in $D$ and $\mathtt{e}_D V=V$ in $D$, the conjugate of \eqref{eq1.prf.lem.A} gives
\[
\begin{split}
    \langle U,V\rangle_D
    &=
    \int_D
    u\cdot
    \Big(
    \lambda\overbar{\varphi}
    -\Delta\overbar{\varphi}
    + (\oprot\overbar{\Psi})\times H_*
    + \nabla\overbar{s_1}
    \Big)
    \dd x \\
    &\quad
    +
    \int_D
    B\cdot
    \Big(
    \lambda\overbar{\Psi}
    + \oprot(\oprot\overbar{\Psi})
    + H_*\cdot\nabla\overbar{\varphi}
    - (\nabla H_*)^\top\overbar{\varphi}
    + \nabla\overbar{s_2}
    \Big)
    \dd x.
\end{split}
\]
On the other hand, testing \eqref{eq0.prf.lem.A} by $(\overbar{\varphi},\overbar{\Psi})$ yields
\[
\begin{split}
    0
    &=
    \int_D
    \Big(
    \lambda u-
    \Delta u
    - H_*\cdot\nabla B
    - B\cdot\nabla H_*
    + \nabla p
    \Big)
    \cdot\overbar{\varphi}
    \dd x \\
    &\quad
    +
    \int_D
    \Big(
    \lambda B
    + \oprot(\oprot B)
    + \oprot(H_*\times u)
    \Big)
    \cdot\overbar{\Psi}
    \dd x.
\end{split}
\]
Subtracting this identity from the previous one, the $\lambda$-terms cancel and we obtain
\[
    \langle U,V\rangle_D
    =
    I_{{\rm Stokes}}
    + I_{{\rm Mag}}
    + I_{{\rm Coup}}
    + I_{{\rm pr},2},
\]
where
\begin{align*}
    I_{{\rm Stokes}}
    &:=
    \int_D
    \Big(
    u\cdot
    (-\Delta\overbar{\varphi}+\nabla\overbar{s_1})
    - 
    (-\Delta u+\nabla p)\cdot\overbar{\varphi}
    \Big)
    \dd x,\\
    I_{{\rm Mag}}
    &:=
    \int_D
    \Big(
    B\cdot\oprot(\oprot\overbar{\Psi})
    -
    \oprot(\oprot B)\cdot\overbar{\Psi}
    \Big)
    \dd x,\\
    I_{{\rm Coup}}
    &:=
    \int_D
    B\cdot
    \Big(
    H_*\cdot\nabla\overbar{\varphi}
    - (\nabla H_*)^\top\overbar{\varphi}
    \Big)
    \dd x \\
    &\quad
    +
    \int_D
    (H_*\cdot\nabla B + B\cdot\nabla H_*)
    \cdot\overbar{\varphi}
    \dd x \\
    &\quad
    +
    \int_D
    u\cdot\big((\oprot\overbar{\Psi})\times H_*\big)
    \dd x
    -
    \int_D
    \oprot(H_*\times u)\cdot\overbar{\Psi}
    \dd x,\\
    I_{{\rm pr},2}
    &:=
    \int_D
    B\cdot\nabla\overbar{s_2}
    \dd x.
\end{align*}
We use the following Green formulas, with $n$ denoting the outward unit normal to $D$:
\begin{align*}
    \int_D a\cdot\nabla\overbar r\dd x
    &=
    -\int_D (\opdiv a)\,\overbar r\dd x
    + \int_{\partial D}(a\cdot n)\,\overbar r\dd \sigma,\\
    \int_D a\cdot\oprot\overbar b\dd x
    &=
    \int_D \oprot a\cdot\overbar b\dd x
    + \int_{\partial D}(a\times n)\cdot\overbar b\dd \sigma,\\
    \int_D \oprot a\cdot\overbar b\dd x
    &=
    \int_D a\cdot\oprot\overbar b\dd x
    + \int_{\partial D}(n\times a)\cdot\overbar b\dd \sigma.
\end{align*}

Since $\opdiv u=\opdiv\varphi=0$ in $D$, two integrations by parts for the Laplacian terms and one integration by parts for the pressure terms give
\begin{equation*}
\begin{split}
    I_{{\rm Stokes}}
    &=
    \int_{\partial D}
    \Big(
    (-pn)\cdot\overbar{\varphi}
    + (u\cdot n) \overbar{s_1}
    - u\cdot\partial_n\overbar{\varphi}
    + \partial_nu\cdot\overbar{\varphi}
    \Big)
    \dd \sigma \\
    &=
    \mathcal{B}_{{\rm Stokes}}(u,p,\varphi).
\end{split}
\end{equation*}
Similarly, the rot Green formulas give
\begin{equation*}
\begin{split}
    I_{{\rm Mag}}
    &=
    \int_D
    B\cdot\oprot(\oprot\overbar{\Psi})\dd x
    -
    \int_D
    \oprot(\oprot B)\cdot\overbar{\Psi}\dd x\\
    &=
    \int_{\partial D}
    \Big(
    (B\times n)\cdot\oprot\overbar{\Psi}
    +(\oprot B\times n)\cdot\overbar{\Psi}
    \Big)
    \dd \sigma \\
    &=
    \mathcal{B}_{{\rm Mag}}(B,\Psi).
\end{split}
\end{equation*}
The coupling part is simplified as follows. By the convention for $(\nabla H_*)^\top$, we have
\[
    B\cdot\big((\nabla H_*)^\top\overbar{\varphi}\big)
    =
    (B\cdot\nabla H_*)\cdot\overbar{\varphi}.
\]
Thus the corresponding two terms cancel. The remaining $\varphi$-coupling is
\begin{equation*}
\begin{split}
    &\int_D
    \Big(
    B\cdot(H_*\cdot\nabla\overbar{\varphi})
    + (H_*\cdot\nabla B)\cdot\overbar{\varphi}
    \Big)
    \dd x \\
    &=
    \int_D
    \opdiv\big((B\cdot\overbar{\varphi})H_*\big)
    \dd x
    =
    \int_{\partial D}
    (H_*\cdot n)(B\cdot\overbar{\varphi})
    \dd \sigma
    =0, 
\end{split}
\end{equation*}
where we have used $\opdiv H_*=0$ in $D$ and $H_*\cdot n=0$ on $\partial D$. For the remaining $\Psi$-coupling, the scalar triple product identity gives
\[
    u\cdot\big((\oprot\overbar{\Psi})\times H_*\big)
    =
    (H_*\times u)\cdot\oprot\overbar{\Psi}.
\]
Therefore, by the rot Green formula,
\begin{equation*}
\begin{split}
    &\int_D
    u\cdot\big((\oprot\overbar{\Psi})\times H_*\big)
    \dd x
    -
    \int_D
    \oprot(H_*\times u)\cdot\overbar{\Psi}
    \dd x \\
    &=
    \int_{\partial D}
    \big((H_*\times u)\times n\big)\cdot\overbar{\Psi}
    \dd \sigma.
\end{split}
\end{equation*}
Since $H_*\cdot n=0$ on $\partial D$, the identity
$(a\times b)\times c=(a\cdot c)b-(b\cdot c)a$ yields
\[
    (H_*\times u)\times n
    =
    (H_*\cdot n)u-(u\cdot n)H_*
    =
    -(u\cdot n)H_*
\]
on $\partial D$. Consequently,
\[
    I_{{\rm Coup}}
    =
    -\int_{\partial D}
    (u\cdot n)(H_*\cdot\overbar{\Psi})
    \dd \sigma
    =
    \mathcal{B}_{{\rm Coup}}(U,W).
\]
Finally, since $\opdiv B=0$ in $D$, the gradient Green formula gives
\begin{equation*}
    I_{{\rm pr},2}
    =
    \int_{\partial D}
    (B\cdot n)\,\overbar{s_2}\dd \sigma
    =
    \mathcal{B}_{{\rm pr},2}(B,s_2).
\end{equation*}
Combining the four identities for $I_{{\rm Stokes}}$, $I_{{\rm Mag}}$, $I_{{\rm Coup}}$, and $I_{{\rm pr},2}$ proves \eqref{eq2.prf.lem.A}.
\end{proofx}

    \section*{Acknowledgment}

A part of this work was done during the authors' stay at the Research Institute for Mathematical Sciences (RIMS), Kyoto University, to attend the RIMS Workshop ``Mathematical Analysis of Viscous Incompressible Fluid”. The authors are grateful for its hospitality.

\end{document}